\documentclass{amsart}
\usepackage{amsmath,amsthm,latexsym,amssymb}
\usepackage{eucal}
\usepackage{mathrsfs}
\usepackage[all]{xy}
\usepackage{color}
\usepackage{graphicx}

\numberwithin{equation}{section}
\newtheorem{thm}{Theorem}[section]
\newtheorem{lem}[thm]{Lemma}
\newtheorem{prop}[thm]{Proposition}
\newtheorem{cor}[thm]{Corollary}

\theoremstyle{definition}
\newtheorem{defn}[thm]{Definition}
\theoremstyle{remark}
\newtheorem{rmk}[thm]{Remark}

\newtheorem{ex}[thm]{Example}
\newtheorem{exs}[thm]{Examples}

\newtheorem{notn}[thm]{Notation}

\newcommand \vp{\varphi}

\newdir{ >}{{}*!/-5pt/@{>}}

\newcommand \tot{\widetilde{\otimes}}

\newcommand\op{\mathcal}
\newcommand\cat{\mathsf}
\newcommand\seq[1]{{#1}\cat{Seq}}
\newcommand\sseq{\seq{\mathfrak S}}
\newcommand\adjunct[4]{\xymatrix{#1\ar @<1.25ex>[rr]^{#3}&\perp&#2\ar @<1.25ex>[ll]^{#4}}}
\newcommand\map{\operatorname{Map}}

\newcommand\bimod[1]{\cat{Bimod}_{#1}}
\newcommand\Bimod[2]{\cat{Bimod}_{#1,#2}}
\newcommand\Biset[2]{\cat{Biset}_{#1,#2}}
\newcommand\id{\text{Id}}
\newcommand\G{\mathbb G}
\newcommand\N{\mathbb N}
\newcommand\R{\mathbb R}
\renewcommand\S{\cat S}
\newcommand\A{\cat A}
\newcommand\F{\cat F}

\begin{document}
\title {The Boardman-Vogt tensor product of operadic bimodules}
\author{William Dwyer}
\address{Department of Mathematics, University of Notre Dame, Notre
  Dame, IN 46556, USA}
\email{dwyer.1@nd.edu}
\author {Kathryn Hess}
\address{MATHGEOM\\
    Ecole Polytechnique F\'ed\'erale de Lausanne \\
    CH-1015 Lausanne \\
    Switzerland}
    \email{kathryn.hess@epfl.ch}

\date \today

 \begin{abstract} We define and study a lift of the Boardman-Vogt tensor product of operads to bimodules over operads. 
 \end{abstract}
 
 \maketitle

 \tableofcontents


\section*{Introduction}

Let $\cat {Op}$ denote the category of symmetric operads  in the monoidal category $\cat S$ of  simplicial sets. The Boardman-Vogt tensor product \cite{boardman-vogt:lnm}
$$-\otimes -: \cat {Op} \times \cat {Op} \to \cat {Op},$$
which endows the category $\cat {Op}$ with a symmetric monoidal structure, codifies interchanging algebraic structures.  For  all $\op P, \op Q\in \cat {Op}$, a $(\op P\otimes \op Q)$-algebra can be viewed as a $\op P$-algebra in the category of $\op Q$-algebras or as a $\op Q$-algebra in the category of $\op P$-algebras.  In this article we lift  the Boardman-Vogt tensor product to the category of composition bimodules over operads and study the properties of the lifted tensor product. 

The lifted Boardman-Vogt tensor product is an essential tool in two articles that we are currently preparing.  One of these articles concerns the space of configurations in a product of framed manifolds, while in the second we generalize \cite{dwyer-hess}, building an ``operadic'' model for the space of long links in $\R^{m}$ for $m\geq 4$.

Let $\op P, \op Q \in \cat {Op}$. Let $\Bimod{\op P}{\op Q}$ denote the category of composition bimodules over $\op P$ on the left and $\op Q$ on the right. An object of $\Bimod{\op P}{\op Q}$ is a symmetric sequence in $\cat S$ endowed with a left action  of $\op P$ and a right action of $\op Q$, with respect to the composition monoidal product $\circ$ of symmetric sequences, which are appropriately compatible. A pair of operad morphisms $\vp:\op P\to \op P'$ and $\psi:\op Q \to \op Q'$ gives rise to a functor $(\vp^{*},\psi^{*}):\Bimod{\op P'}{\op Q'}\to \Bimod{\op P}{\op Q}$ by restriction of coefficients. 

Gathering together all composition bimodules over all operads, we form a category $\cat {Bimod}$. An object of $\cat {Bimod}$ is a composition bimodule over a pair of operads $(\op P,\op Q)$.  A morphism in $\cat {Bimod}$ from a $(\op P, \op Q)$-bimodule $\op M$ to a $(\op P', \op Q')$-bimodule $\op M'$ consists of a triple $(\vp,\psi, f)$, where $\vp:\op P\to \op P'$ and $\psi:\op Q \to \op Q'$ are operad morphisms, and $f: \op M \to (\vp^{*},\psi^{*})(\op M')$ is a morphism of $(\op P,\op Q)$-bimodules.  There is an obvious projection functor $\Pi: \cat {Bimod}\to \cat {Op}\times \cat {Op}$, which admits a section $$\Gamma: \cat {Op}\times\cat {Op} \to \cat {Bimod}: (\op P, \op Q)\mapsto \op P\circ \op Q,$$
where $\op P\circ \op Q$ is viewed as the free $(\op P,\op Q)$-bimodule on a generating symmetric sequence that is a singleton concentrated in arity 1.

In this article we construct a functor 
$$-\tot-: \cat {Bimod}\times \cat {Bimod} \to \cat {Bimod}$$
such that both $\Pi$ and $\Gamma$ are strictly monoidal with respect to $\tot$ and $\otimes$ (Theorem \ref{thm:tot}).  In other words, if  $\op M$ is a $(\op P,\op Q)$-bimodule and $\op M'$ is a $(\op P',\op Q')$-bimodule, then $\op M\tot \op M'$ is an $(\op P\otimes \op P', \op Q\otimes \op Q')$-bimodule, and $(\op P\circ \op Q)\tot (\op P'\circ \op Q')=(\op P\otimes \op P') \circ (\op Q\otimes \op Q')$.

Underlying the lift $-\tot-$ of the Boardman-Vogt tensor product is an intriguing closed, symmetric monoidal structure on the category of symmetric sequences, which we call the \emph{matrix monoidal structure} (Section \ref{sec:mms}).  To the best of our knowledge, this monoidal structure has not appeared before in the literature, though it arises quite naturally.  We provide an explicit description of the adjoint hom-functor (Proposition \ref{prop:adjunction}), in terms of a certain endofunctor on the category of symmetric sequences, which we call the \emph{divided powers functor} (Definition \ref{defn:divpow}).  The key to constructing the lift $-\tot-$ is the existence of a well-behaved natural transformation that intertwines the matrix monoidal structure and the composition monoidal structure on the category of symmetric sequences in a way compatible with the multiplication on the Boardman-Vogt operad (Proposition \ref{prop:sigma} and Theorem \ref{thm:tau}).

As the divided powers functor is interesting in itself, we also elaborate on its relationship to standard operadic algebra.  We show, in particular, that the divided powers functor preserves operadic bimodule structure (Proposition \ref{prop:bimod}), which enables us to prove that the monoidal structure on $\cat {Bimod}$ is closed (Proposition \ref{prop:closed}).  The divided powers functor also converts operads endowed with an axial structure (Definition \ref{defn:axial}) into nonunital operads (Proposition \ref{prop:axial}).  We provide several examples of such operads. Finally, we observe that the divided powers functor is monoidal with respect to the levelwise and graded monoidal structures on symmetric sequences (Proposition \ref{prop:grmon}).

We prove Proposition \ref{prop:sigma} and Theorem \ref{thm:tau} in the appendix, as their proofs are rather long and technical and might distract the reader from the actual construction of the lifted Boardman-Vogt tensor product, if left in the main body of the text. These proofs involve a coordinate-free approach to symmetric sequences and their various monoidal structures, which may be of independent interest.

\subsection*{Conventions}

If $A$ and $B$ are objects in a small category $\cat C$, then $ \cat C(A,B)$ denotes the set of morphisms in $\cat C$ with source $A$ and target $B$.   We often denote the identity morphism on an object $A$ also by $A$.   

We denote the category of simplicial sets by $\S$.  If $G$ is any group, then $\S_{G}$ denotes the category of simplicial sets endowed with a simplicial $G$-action and of $G$-equivariant simplicial morphisms.  

We denote the symmetric group on $n$ letters by $\mathfrak S_{n}$.

To limit confusion, we reserve the symbol $\circ$ for the composition product of (symmetric) sequences to as great an extent as possible, denoting composition of composable morphisms $f:A\to B$ and $g:B\to C$ simply $gf:A\to C$.

\section{Lifting the Boardman-Vogt tensor product}

Throughout this section we work in the category $\cat S$ of simplicial sets, endowed with its cartesian monoidal structure.  Similar constructions and arguments work when the underlying category is that of compactly generated Hausdorff spaces with its usual monoidal structure.

\subsection{The Boardman-Vogt tensor product of operads}

Before recalling the Boardman-Vogt tensor product of operads, we introduce some useful notation.  We assume that the reader is familiar with the usual composition monoidal structure on the category of symmetric sequences in a cocomplete monoidal category, denoted $\circ$ in this article, and with operads, the monoids in this category.  Relevant introductory references include \cite{fresse}, \cite{loday-vallette}, and  \cite{markl-shnider-stasheff}.  We recall that the category $\cat {Op}$ of operads in simplicial sets is cocomplete; see, e.g., \cite[Section 4]{dwyer-hess} for a discussion of coproducts in $\cat {Op}$ and their properties.

\begin{notn}  For any two (symmetric) sequences $\op X=\big(\op X(n)\big)_{n\geq 0}$ and $\op Y=\big(\op Y(n)\big)_{n\geq 0}$ of simplicial sets, a representative of a typical element of arity $n$ in the composition product $\op X \circ \op Y$ of the two sequences  is denoted $(x; y_{1},...,y_{k};\tau)$, where $x\in \op X(k)$ and $y_{i}\in \op Y(n_{i})$ for $1\leq i\leq k$, with $\sum _{i}n_{i}=n$, and $\tau \in \mathfrak S_{n}$.   The right action of $\mathfrak S_{n}$ on $\op X\circ \op Y$ is given by
$$(x; y_{1},...,y_{k};\id)\cdot \tau=(x; y_{1},...,y_{k};\tau).$$

Recall that if $\tau_{i}\in \mathfrak S_{n_{i}}$ for $1\leq i\leq k$, and $n=\sum _{i=1}^{k}n_{i}$, then $\tau _{1}\oplus \cdots \oplus \tau _{k}\in\mathfrak S_{n}$ is the permutation ``block by block'' specified by
$$(\tau _{1}\oplus \cdots \oplus \tau _{k})(l)=\tau _{j}\big(l-\sum_{i=1}^{j-1}n_{i}\big) + \sum_{i=1}^{j-1}n_{i}\quad\text{for all}\quad  \sum_{i=1}^{j-1}n_{i}<l\leq  \sum_{i=1}^{j}n_{i}.$$
The equivalence relation on  representatives of elements of $\op X \circ \op Y$ satisfies
$$(x; y_{1},...,y_{k};\tau _{1}\oplus \cdots \oplus \tau _{k})\sim (x; y_{1}\cdot \tau_{1}^{-1}, \cdots, y_{k}\cdot \tau_{k}^{-1};\id)$$
and
$$(x\cdot \sigma^{-1}; y_{1},...,y_{k};\id)\sim (x; y_{\sigma(1)},...,y_{\sigma(k)};\id)$$
for all $\sigma \in \mathfrak S_{k}$ and $\tau_{i}\in \mathfrak S_{n_{i}}$ for $1\leq i\leq k$, where $x\in \op X(k)$ and $y_{i}\in \op Y(n_{i})$ for $1\leq i\leq k$.

If $\op P$ is an operad with multiplication map $\mu:\op P\circ \op P\to \op P$, and $p\in \op P(k)$, $p_{i}\in \op P(n_{i})$ for $1\leq i\leq k$ with $\sum _{i}n_{i}=n$, then we write 
$$p(p_{1},...,p_{k}):=\mu (p;p_{1},...,p_{k};\id)\in \op P(n).$$
Note that since $\mu$ is equivariant, it is specified by its values on elements of $\op P\circ \op P$ with representatives of the form $(p;p_{1},...,p_{k};\id)$
\end{notn}

\begin{defn}\cite{boardman-vogt:lnm}, \cite{dunn} The \emph{Boardman-Vogt tensor product} of operads $\op P$ and $\op Q$ is the operad $\op P\otimes \op Q$ that is the quotient of the coproduct $\op P\coprod \op Q$ of operads by the equivalence relation generated by
$$(p;\underbrace{q,..,q}_{k};\id) \sim (q;\underbrace{p,...,p}_{l};\tau_{k,l})$$
for all $p\in \op P(k)$ and $q\in \op Q(l)$, where $\tau_{k,l}\in \mathfrak S _{kl}$ is the transpose permutation that ``exchanges rows and columns'', i.e., for all $1\leq m=(i-1)l+j\leq kl$, where $1\leq i\leq k$ and $1\leq j\leq l$, 
$$\tau_{k,l}(m)=(j-1)k+i.$$
\end{defn}

\begin{notn} We let $p\otimes q$ denote the common equivalence class of  
$$(p;\underbrace{q,..,q}_{k};\id)\quad\text{and}\quad(q;\underbrace{p,...,p}_{l};\tau_{k,l})$$ in $(\op P\otimes \op Q)(kl)$
for all $p\in \op P(k)$ and $q\in \op Q(l)$.
\end{notn}

\centerline{\includegraphics [scale=1.8]{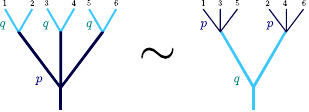}}

\begin{rmk}\label{rmk:switch} In terms of the notation above, if $p\in \op P(k)$ and $q\in \op Q(l)$, then 
$$p\otimes q = q\otimes p \cdot \tau_{k,l}.$$
\end{rmk}

\begin{rmk}  Given two symmetric operads, one is naturally led to consider the map that sends a pair $(p,q)$ with $p\in \op P(k)$ and $q\in \op Q(l)$ to $p\otimes q\in (\op P\otimes \op Q)(kl)$.  For this map to respect arity, we need to define a  monoidal product on symmetric sequences with respect to which the pair $(p,q)$ is in arity $kl$.  This observation is the key to our lift of the Boardman-Vogt tensor product in section \ref{sec:bvt}, of which we give an indication now.

Suppose that $\op X$ and $\op Y$ are left modules (with respect to the composition monoidal product $\circ$) over the operads $\op P$ and $\op Q$, respectively.  Roughly speaking, the left $(\op P\otimes \op Q)$-module $\op X \tot \op Y$ is generated by elements of the form $x\square y$, where  the arity of $x\square y$ is the product of the arities of $x$ and $y$.  Moreover, for $p\in \op P(k)$ and $q\in \op Q(l)$, relations like 
$$p(x_{1}\square y,..., x_{k}\square y)=p(x_{1},...,x_{k})\square y) \quad\text{ and }\quad q(x\square y_{1},..., x\square y_{l})=x\square q(y_{1},...,y_{l})$$
hold, along with appropriate relations involving symmetric groups.
\end{rmk}

\subsection{Sequences concentrated in arity 1}

Before explaining how to lift the Boardman-Vogt tensor product in the general case, we analyze the construction in the special case of sequences $\op X$ concentrated in arity 1, i.e., $\op X(n)= \emptyset$ if $n\not=1$. Operads concentrated in arity 1 are nothing but simplicial monoids, while operadic bimodules concentrated in arity 1 are just simplicial bisets.

If $P$ and  $Q$ are simplicial monoids, seen as operads concentrated in arity 1, then their Boardman-Vogt tensor product is  their cartesian product: $P\otimes Q=P\times Q$, endowed with its usual componentwise multiplicative structure.  If $M$ and $M'$ are a simplicial $(P,Q)$-biset and a simplicial $(P',Q')$-biset, respectively, then $M\times M'$ is naturally a simplicial $(P\times P', Q\times Q')$-biset, endowed with componentwise left and right actions.  The lift of the Boardman-Vogt tensor product is thus completely trivial in the case of sequences concentrated in arity 1: $M\tot M'=M\times M'$, endowed with the componentwise actions.  To see how to generalize this lifting to general symmetric sequences, however, we must look more closely at what is happening in this easy case. 

For any simplicial monoids $P$ and $Q$, let $F_{P,Q}: \cat S \to \Biset PQ$ denote the free $(P,Q)$-biset functor, i.e., $F_{P,Q}(X)= P\times X\times Q$, endowed with the obvious left $P$-action and right $Q$-action.  Observe that for all simplicial sets $X$ and $Y$ and all simplicial monoids $P$, $P'$, $Q$ and $Q'$, there is a natural simplicial isomorphism of $(P\otimes P', Q\otimes Q')$-bisets
\begin{align*}
F_{P,Q}(X)\tot F_{P',Q'}(X') &\to F_{P\times P', Q\times Q'}(X\times X')\\
\big ((p,x,q), (p',x',q')\big) &\mapsto \big((p,p'), (x,x'), (q,q')\big),
\end{align*}
with underlying simplicial map
\begin{equation}\label{eqn:transf} \upsilon_{X,X'}: (P\times X\times Q)\times (P'\times X'\times Q')\to (P\times P')\times (X\times X')\times (Q \times Q').
\end{equation} 
The existence of this natural map implies that every pair of morphisms of bisets
$$a: F_{P,Q}(X)\to F_{P,Q}(Y)\quad\text{and}\quad a': F_{P',Q'}(X')\to F_{P',Q'}(Y')$$ 
 induces a morphism of bisets
 $$a\tot b: F_{P\times P', Q\times Q'}(X\times X') \to F_{P\times P', Q\times Q'}(Y\times Y'),$$
 which is determined on the generators by the simplicial map
 $$X\times X' \xrightarrow {a^{\sharp}\times b^{\sharp}} (P\times Y\times Q)\times (P'\times Y'\times Q') \xrightarrow {\upsilon_{X,X'}} (P\times P')\times (Y\times Y')\times (Q\times Q'),$$
 where $a^{\sharp}$ and $b^{\sharp}$ are the simplicial maps given by restriction of $a$ and $b$ to the generators.
 The tensor product of arbitrary bisets $M$ and $M'$ can then be constructed as a colimit of free bisets, since every biset is a coequalizer of free bisets.   It is easy, and a good warm up for the general case, to check that via this more complicated approach, one obtains $M\tot M'\cong M\times M'$, as expected.
 
We next generalize this approach to the case of arbitrary sequences.  For sequences concentrated in arity 1, the Boardman-Vogt tensor product of operads, the composition product of sequences and the cartesian product of generating simplicial sets all coincide.  In the general case, the three monoidal structures in play differ greatly.  Since the Boardman-Vogt tensor product and the composition product are well known, we need only to determine what monoidal structure on symmetric sequences should play the role of the cartesian product of generating simplicial sets, which is the goal of the next section.

\subsection{The matrix monoidal structure}\label{sec:mms}

Let $\G$ be a totally disconnected groupoid with set of objects $\mathbb N$. Let $G_{n}=\G(n,n)$, with neutral element $e_{n}$.  Suppose that $\G$ admits a monoidal structure $\nu: \G \times \G \to \G$ that is multiplicative on objects, i.e., $\nu (m,n)=mn$.  In particular, for every $m,n\in \mathbb N$, there is a homomorphism
$$\nu_{m,n}:G_{m}\times G_{n}\to G_{mn},$$
and
$$\nu_{lm,n}(\nu_{l,m}\times \id_{n})=\nu_{l,mn}(\id_{l}\times \nu_{m,n}):G_{l}\times G_{m}\times G_{n}\to G_{lmn}$$
for all $l,m,n\in \mathbb N$. 

\begin{exs}\label{exs} Of course the groupoid $\mathbb N$, where each component group is trivial, provides an example of such a monoidal groupoid. Another important example is the symmetric groupoid $\mathfrak S$, where
$\nu_{m,n}:\mathfrak S_{m}\times \mathfrak S_{n}\to \mathfrak S_{mn}$ is defined by
$$\nu_{m,n}(\sigma, \tau)(i,j)=\big(\sigma(i), \tau (j)\big).$$
Here we view $\mathfrak S _{mn}$ as the group of permutations of $\{1,...,m\}\times \{1,...,n\}$, by identifying the pair $(i,j)$ with the number $(i-1)n+j$ for all $1\leq i\leq m$ and $1\leq j\leq n$.

The braid groupoid $\mathbb B$ and the pure braid groupoid $\mathbb P$ admit similar monoidal structures.
\end{exs}

Let $\seq{\G}$ denote the category of $\G$-sequences of simplicial sets, i.e., of functors from $\G^{op}$ to $\cat{S}$. Objects of $\seq{\G}$ consist of sequences $\op X=\big(\op X(n)\big)_{n\geq 0}$ of objects in $\cat {S}$ together with actions $\op X(n)\times G_{n}\to \op X(n)$ for all $n\geq 0$.  The homomorphisms $\nu_{m,n}$ induce actions of $G_{m}$ and $G_{n}$ on $\op X(mn)$ for all $m$ and $n$, given by
$$\op X(mn) \times G_{m}\to \op X(mn): (x,a)\mapsto x \cdot \nu_{m,n}(a,e_{n})$$
and
$$\op X(mn) \times G_{n}\to \op X(mn): (x,b)\mapsto x \cdot \nu_{m,n}(e_{m},b).$$
Note that the actions of $G_{m}$ and $G_{n}$ on $\op X(mn)$ commute, since for all $a\in G_{m}$ and $b\in G_{n}$,
$$\nu_{m,n}(a,e_{n})\nu_{m,n}(e_{m},b)=\nu_{m,n}(a,b)=\nu_{m,n}(e_{m},b)\nu_{m,n}(a,e_{n}),$$
as $\nu_{m,n}$ is a homomorphism.

Recall that for any monoidal category $(\cat V, \wedge, I)$, Day convolution \cite{day} gives rise to a closed monoidal structure on the presheaf category $\cat{Set}^{\cat V^{op}}$ with respect to which the Yoneda embedding $\cat V \to \cat{Set}^{\cat V^{op}}$ is strongly monoidal. Applying Day convolution to $\nu: \G \times \G \to \G$, we obtain a monoidal product
$$-\square -: \seq{\G} \times \seq{\G} \to \seq{\G},$$
which we call the \emph{matrix monoidal structure}, specified by
$$(\op X \square \op Y)(n)=\coprod_{lm=n}\big(\op X(l)\times \op Y(m)\big)\times _{G_{l}\times G _{m}} G _{n},$$
where the action of $G_{l}\times G_{m}$ on $G_{n}$ is given by $\nu_{l,m}$.   The unit for this monoidal structure is the same as the unit for the composition monoidal structure, i.e., the sequence $\op J$ with $\op J(1)$ a singleton and $\op J(n)$ empty for all $n\not =1$.

To describe the right adjoint to $-\square \op Y$, for a fixed $\G$-sequence $\op Y$, we need the following definition.

\begin{defn}\label{defn:divpow} The \emph{$n^{\text{th}}$-divided powers functor} associated to $(\G, \nu)$, 
$$\gamma^{\G}_{n}: \seq{\G} \to \seq{\G},$$ 
is defined on objects by
$$\gamma^{\G}_{n}(\op X)(m)=\nu_{m,n}^{*}\op X(mn)^{},$$
i.e., the $G_{mn}$-action on $\op X(mn)$ is pulled back by $\nu_{m,n}$ to an action of $G_{m}\times G_{n}$ and then implicity to an action of $G_{m}$.  

The \emph{graded divided powers functor} is defined by
$$\gamma^{\G}_{\bullet}:\seq{\G} \to (\seq{\G})^{\mathbb N}:\op X\mapsto \big(\gamma^{\G}_{n}(\op X)\big)_{n\geq 0}.$$
\end{defn}

\begin{rmk}\label{rmk:action} Since $\gamma^{\G}_{n}(\op X)(m)=\nu_{m,n}^{*} \op X(mn)$, it admits a natural $G_{n}$-action, for all $m$. It follows that for all $n$, $\gamma^{\G}_{n}(\op X)$ is a sequence in $\S_{G_{n}}$.
\end{rmk}

\begin{notn} When $\G=\N$, we denote suppress $\N$ from the notation for the (graded) divided powers functor.  
\end{notn}

Unraveling Day's constructions, we obtain the following formula for the internal hom corresponding to the monoidal product $-\square -$ on $\seq{\G}$.  In the proof, we check the adjunction by hand, to clarify the roles of the various group actions, without appealing to the formalism of Day convolution.

\begin{notn}   For any two $\G$-sequences $\op X$ and $\op Y$, let $\map_{\G} (\op X, \op Y)$ denote the simplicial mapping space, i.e.,
$$\map_{\G}(\op X, \op Y)=\prod_{n} \operatorname{map}_{G_{n}} \big(\op X(n), \op Y(n)\big),$$
where $\operatorname{map}_{G_{n}}$ denotes the usual simplicial mapping space of equivariant maps between simplicial $G_{n}$-sets.  
\end{notn}

\begin{rmk}  For any $\G$-sequences $\op Y$ and $\op Z$, it follows from Remark \ref{rmk:action} that the sequence $\Big( \map_{\G}\big(\op Y, \gamma^{\G}_{n}(\op Z)\big)\Big)_{n\geq 0}$ of simplicial sets is itself a $\G$-sequence.  
\end{rmk}

\begin{prop}\label{prop:adjunction} For every $\op Y\in \seq{\G}$, there is an adjunction
$$\adjunct{\seq{\G}}{\seq{\G}}{-\square \op Y}{\map_{\G}\big(\op Y, \gamma^{\G}_{\bullet}(-)\big)},$$
which is natural in $\op Y$.
\end{prop}

\begin{proof} The proposition is a consequence of the following sequence of natural isomorphisms.  If $\op X$, $\op Y$, and $\op Z$ are $\G$-sequences, then
\begin{align*}
\seq{\G}(\op X\square \op Y, \op Z) &= \prod_{n\geq 0}\S_{G_{n}}\big(\op X\square\op Y(n), , \op Z(n)\big)\\
&=\prod_{n\geq 0}\S _{G_{n}}\big (\coprod_{lm=n}\op X(l)\times Y(m)\times_{G_{l}\times G_{m}}G_{n}, \op Z(n)\big)\\
&\cong\prod_{n\geq 0}\prod_{lm=n}\S_{G_{l}\times G_{m}}\big(\op X(l)\times\op Y(m),  \nu_{l,m}^{*}\op Z(n)\big)\\
&\cong\prod_{l\geq 0}\S_{G_{l}}\bigg(\op X(l), \prod _{m\geq 0} \operatorname{map}_{G_{m}}\big(\op Y(m), \nu_{l,m}^{*}\op Z(ml)\big)\bigg)\\
&=\seq{\G}\big(\op X, \map_{\G}(\op Y, \gamma^{\G}_{\bullet}(\op Z)\big).
\end{align*}
\end{proof}

\subsection{The Boardman-Vogt tensor product of bimodules}\label{sec:bvt}

Having defined and studied the matrix monoidal structure on symmetric sequences, we are ready to lift the Boardman-Vogt tensor product of operads to bimodules over operads.

\begin{notn} For operads $\op P$ and $\op Q$ with multiplication maps $\mu$ and $\nu$, let 
$$F_{\op P, \op Q}: \sseq\to \Bimod{\op P} {\op Q}: \op X \mapsto (\op P\circ \op X\circ \op Q, \mu\circ\id_{\op X\circ \op Q}, \id_{\op P\circ \op X}\circ \nu)$$ 
denote the free $(\op P,\op Q)$-bimodule functor. 
\end{notn}

\begin{thm}\label{thm:tot} Let $\op P$, $\op P'$, $\op Q$ and $\op Q'$ be symmetric operads.  There exists a functor
$$-\tot-: \bimod{\op P,\op Q}\times \bimod {\op P',\op Q'} \to \bimod {\op P\otimes \op P', \op Q\otimes \op Q' },$$
natural in $\op P$, $\op P'$, $\op Q$ and $\op Q'$, such that 
$$F_{\op P, \op Q}(\op X)\tot F_{\op P', \op Q'}(\op X') = F_{\op P\otimes \op P', \op Q\otimes \op Q'}(\op X \square \op X')$$
for all $\op X, \op Y\in \sseq$. 
\end{thm}

\begin{rmk} By restricting to the special cases where either $\op P'=\op Q'=\op J$ or $\op P=\op Q=\op J$, where $\op J$ is the unit symmetric sequence, we obtain as a consequence of Theorem \ref{thm:tot} that the Boardman-Vogt tensor product lifts to both left and right modules over operads.
\end{rmk}

Proving Theorem \ref{thm:tot} requires a good understanding of the relationship between the composition and matrix monoidal structures on $\seq{\mathfrak S}$.  To this end, we establish a sequence of propositions describing this relationship, then apply these propositions to proving Theorem \ref{thm:tot}.

Inspired by our analysis of the case of sequences concentrated in arity 1, we begin by analyzing the case of free bimodules.

 \begin{notn} Let 
$$\Phi_{\op P, \op Q}:(\sseq)^{op}\times \sseq \to \cat {Set}$$
denote the functor defined on objects by 
$$\Phi_{\op P, \op Q}(\op X, \op Y)= \Bimod{\op P}{\op Q} (F_{\op P,\op Q}\op X, F_{\op P,\op Q}\op Y).$$
\end{notn}

As stated precisely below, $\Phi_{(-,-)}$ is a family of functors ``monoidally parametrized'' by $(\cat {Op},\otimes)$, which is the key to our definition of the tensor product of bimodule morphisms.

\begin{prop}\label{prop:tot} For all operads $\op P$, $\op P'$, $\op Q$ and $\op Q'$, the diagram of functors
$$\xymatrix{\big((\sseq)^{op}\times \sseq\big)\times \big((\sseq)^{op}\times \sseq\big)\ar[d]_{(23)} \ar[rr]^(0.70){\Phi_{\op P, \op Q}\times \Phi_{\op P',\op Q'}}&&\cat {Set}\times \cat {Set}\ar [dd]^{-\times-}\\
\big((\sseq)^{op}\times (\sseq)^{op}\big)\times \big(\sseq\times \sseq\big)\ar[d]_{(-\square-)\times (-\square-)}\\
(\sseq)^{op}\times \sseq\ar [rr]^{\Phi_{\op P\otimes\op P', \op Q\otimes \op Q'}}&&\cat {Set}}$$
commutes up to a natural transformation 
$$\xi: (-\times -)\circ (\Phi_{\op P,\op Q}\times \Phi_{\op P',\op Q'}) \Longrightarrow \Phi_{\op P\otimes\op P', \op Q\otimes \op Q'}\circ\big( (-\square-)\times (-\square-)\big)\circ (23).$$
\end{prop}

\bigskip

\begin{rmk} The existence of the natural map
$$\xi: \Phi_{\op P,\op Q}(\op X, \op Y)\times \Phi_{\op P', \op Q'}(\op X', \op Y')\to  \Phi_{\op P\otimes \op P', \op Q \otimes \op Q'}(\op X\square \op X', \op Y\square \op Y')$$ 
implies that for every pair of bimodule morphisms $a:F_{\op P,\op Q}\op X \to F_{\op P,\op Q}\op Y$ and $b: F_{\op P',\op Q'}\op X'\to F_{\op P',\op Q'}\op Y'$, i.e., elements of $\Phi_{\op P,\op Q}(\op X, \op Y)$ and $\Phi_{\op P',\op Q'}(\op X', \op Y')$, respectively, there is a bimodule morphism
$$ \xi(a,b): F_{\op P\otimes \op P', \op Q \otimes \op Q'}(\op X\square \op X')\to F_{\op P\otimes \op P', \op Q \otimes \op Q'}(\op Y\square \op Y').$$
\end{rmk}
\medskip

It will follow from the construction of $\xi$ that it preserves composition in both variables, as formulated precisely below.
\medskip

\begin{prop}\label{prop:composition} For all bimodule morphisms $a:F_{\op P,\op Q}\op X \to F_{\op P, \op Q}\op Y$, $b:F_{\op P, \op Q}\op Y\to F_{\op P, \op Q}\op Z$, $a':F_{\op P',\op Q'}\op X' \to F_{\op P', \op Q'}\op Y'$, and $b':F_{\op P', \op Q'}\op Y'\to F_{\op P', \op Q'}\op Z'$,
$$\xi(b, b')\xi(a, a')=\xi(ba, b'a'): F_{\op P\otimes \op P', \op Q\otimes \op Q'}(\op X \square \op X') \to F_{\op P\otimes \op P', \op Q\otimes \op Q'}(\op Z \square \op Z').$$
\end{prop}
\medskip

An important tool in the construction of the natural transformation $\xi$ is  the following transformation that intertwines $\square$ and $\circ$, generalizing the transformation (\ref{eqn:transf}).

\begin{prop}\label{prop:sigma} For all $\op V,\op W, \op Y, \op Z\in \seq{\mathfrak S}$, there is a natural morphism of $\mathfrak S$-sequences
$$\sigma: (\op V\circ \op W)\square(\op Y\circ \op Z) \to (\op V \square \op Y) \circ (\op W \square \op Z)$$
such that 
$$\xymatrix{(\op U\circ\op V\circ \op W)\square(\op X\circ\op Y\circ \op Z)\ar [r]^{\sigma}\ar [d]_{\sigma}&(\op U\square \op X)\circ\big((\op V\circ \op W)\square(\op Y\circ \op Z)\big)\ar [d]^{\id\circ \sigma}\\ 
\big((\op U\circ\op V)\square(\op X\circ\op Y)\big)\circ (\op W \square \op Z)\ar [r]^{\sigma\circ \id}&(\op U\square \op X)\circ (\op V \square \op Y) \circ (\op W \square \op Z)}$$
commutes.
\end{prop}

\begin{rmk} It is probably true that $(\seq{\mathfrak S}, \square, \op J, \circ, \op J)$ is a \emph{2-monoidal category} in the sense of Aguiar and Mahajan \cite {aguiar-mahajan}, but as we do not need the full power of this structure in this article, we leave the proof to the interested reader.
\end{rmk}

From the natural transformation $\sigma$ of Proposition \ref{prop:sigma} we can build another natural transformation, for operads this time, which is compatible with the multiplication on the Boardman-Vogt tensor product in a crucial way. 

\begin{thm}\label{thm:tau} Let $\op P$, $\op P'$, $\op Q$ and $\op Q'$ be operads, and let $\op X$ and $\op X'$ be symmetric sequences.  There is a morphism of symmetric sequences
$$\upsilon: (\op P\circ \op X\circ \op Q) \square (\op P'\circ \op X'\circ \op Q') \to (\op P \otimes \op P') \circ (\op X \square \op X') \circ (\op Q\otimes \op Q')$$
that is natural in all variables.  Moreover,  $\upsilon$ is compatible with operad multiplication, in the sense that
{\small{\begin{equation}\label{eqn:tau}\xymatrix{(\op P^{\circ 2}\circ \op X\circ \op Q^{\circ 2})\square \big((\op P')^{\circ 2}\circ \op X'\circ (\op Q')^{\circ 2}\big)\ar[r]^(0.45){\upsilon}\ar [dd]_{(\mu\circ \op X\circ \mu)\square (\mu\circ \op X'\circ \mu)}&(\op P\otimes \op P')\circ \big((\op P\circ \op X\circ \op Q)\square (\op P'\circ \op X'\circ \op Q')\big)\circ (\op Q\otimes \op Q')\ar [d]^{\id_{\op P\otimes \op P'}\circ \upsilon\circ\id_{\op Q\otimes \op Q'} }\\
&(\op P\otimes \op P')^{\circ 2}\circ (\op X\square \op X') \circ (\op Q\otimes \op Q')^{\circ 2}\ar[d]^{\mu\circ \id _{\op X\square \op X'}\circ \mu}\\
(\op P\circ \op X\circ \op Q)\square (\op P'\circ \op X'\circ \op Q')\ar [r]_{\upsilon}&(\op P\otimes \op P')\circ (\op X\square \op X') \circ (\op Q\otimes \op Q')
}
\end{equation}}}

\noindent always commutes.  Here all operad multiplication maps are denoted $\mu$.
\end{thm}

In order not to interrupt the flow of the argument here, we refer the reader to the appendix for the proofs of Proposition \ref{prop:sigma} and Theorem \ref{thm:tau} and continue now with their application to proving Propositions \ref{prop:tot} and \ref{prop:composition} and finally Theorem \ref{thm:tot}.

\begin{proof}[Proof of Proposition \ref{prop:tot}] Let $a\in \Phi_{\op P,\op Q}(\op X, \op Y)$ and $b\in \Phi_{\op P', \op Q'}(\op X', \op Y')$, and let
$$a^{\sharp}: \op X \to \op P\circ \op Y\circ \op Q\quad \text{and} \quad b^{\sharp}: \op X' \to \op P'\circ \op Y'\circ \op Q'$$
be the morphisms of symmetric sequences corresponding to $a$ and $b$ under the free bimodule/forgetful-adjunction. The composite morphism of symmetric sequences
$$\op X \square \op X' \xrightarrow {a^{\sharp}\square b^{\sharp}} (\op P\circ \op Y\circ \op Q) \square (\op P'\circ \op Y'\circ \op Q)' \xrightarrow\upsilon (\op P \otimes \op P') \circ (\op Y \square \op Y') \circ (\op Q\otimes \op Q')$$
corresponds under the free bimodule/forgetful-adjunction to a morphism 
$$\big(\upsilon(a^{\sharp}\square b^{\sharp})\big)^{\flat}: F_{\op P\otimes \op P', \op Q\otimes \op Q'}(\op X \square\op X') \to  F_{\op P\otimes \op P', \op Q\otimes \op Q'}(\op Y\square\op Y')$$
of $(\op P\otimes \op P', \op Q\otimes \op Q')$-bimodules.  

Define 
$$\xi: \Phi_{\op P,Q}(\op X, \op Y)\times \Phi_{\op P',\op Q'}(\op X', \op Y')\to  \Phi_{\op P\otimes \op P', \op Q\otimes \op Q'}(\op X\square \op X', \op Y\square \op Y')$$ 
by $\xi (a,b)=\big(\upsilon(a^{\sharp}\square b^{\sharp})\big)^{\flat}$.  This construction is clearly natural in all variables.
\end{proof}

\begin{proof}[Proof of Proposition \ref{prop:composition}] It is enough to check that $\xi(b, b')\xi(a, a')$ and $\xi(ba, b'a')$ agree on generators, i.e., that 
$$\Big(\xi(b, b')\xi(a, a')\Big)^{\sharp}=\Big(\xi(ba, b'a'))\Big)^{\sharp}: \op X\square \op X'\to F_{\op P\otimes \op P', \op Q\otimes \op Q'}(\op Z\square \op Z').$$
On the one hand, $\Big(\xi(ba, b'a')\Big)^{\sharp}$ is equal to the composite
{\small$$\xymatrix{\op X\square\op X'\ar[r]^(0.27){a^{\sharp}\square (a')^{\sharp}}&
		(\op P\circ \op Y\circ \op Q) \square (\op P'\circ \op Y'\circ \op Q')\ar [rr]^(0.45){\big(b^\sharp\square(b')^\sharp\big)^{\natural}}&&(\op P^{\circ 2}\circ \op Z\circ \op Q^{\circ 2}) \square ((\op P')^{\circ 2}\circ \op Z'\circ (\op Q')^{\circ 2})\ar [d]^{(\mu\circ \id\circ \mu)\square(\mu\circ \id\circ \mu)}\\
		&&&(\op P\circ \op Z\circ \op Q) \square (\op P'\circ \op Z'\circ \op Q')\ar [d]^{\upsilon}\\
		&&&(\op P\otimes \op P')\circ (\op Z\square \op Z')\circ (\op Q\otimes \op Q')}$$}
where $\big(b^\sharp\square(b')^\sharp\big)^{\natural}=(\id\circ (b)^\sharp\circ \id)\square(\id\circ (b')^\sharp\circ \id).$ 

On the other hand,  $\Big(\xi(b, b')\xi(a, a')\Big)^{\sharp}$ is equal to the following composite.
{\small$$\xymatrix{\op X\square\op X'\ar[r]^(0.27){a^{\sharp}\square (a')^{\sharp}}&
		(\op P\circ \op Y\circ \op Q) \square (\op P'\circ \op Y'\circ \op Q')\ar [r]^{\upsilon}&(\op P\otimes \op P')\circ (\op Y\square \op Y')\circ (\op Q\otimes \op Q')\ar [d]^{\id\circ \big(b^{\sharp}\square(b')^{\sharp}\big)\circ \id}\\
		&&(\op P\otimes \op P')\circ \big((\op P\circ \op Z\circ \op Q)\square (\op P'\circ \op Z'\circ \op Q')\big)\circ (\op Q\otimes \op Q')\ar [d]^{\id\circ \upsilon \circ \id}\\
		&&(\op P\otimes \op P')^{\circ 2} \circ(\op Z\square \op Z')\circ (\op Q\otimes \op Q')^{\circ 2} \ar [d]^{\mu\circ\id\circ \mu}\\
		&&(\op P\otimes \op P')\circ (\op Z\square \op Z')\circ (\op Q\otimes \op Q')}$$}

Since $\upsilon$ is a natural transformation,
{\small $$\xymatrix{(\op P\circ \op Y\circ \op Q) \square (\op P'\circ \op Y'\circ \op Q')\ar [rr]^(0.45){\big(b^\sharp\square(b')^\sharp\big)^{\natural}}\ar [d]^\upsilon&&(\op P^{\circ 2}\circ \op Y\circ \op Q^{\circ 2}) \square ((\op P')^{\circ 2}\circ \op Y'\circ (\op Q')^{\circ 2})\ar [d]^\upsilon\\
(\op P\otimes \op P')\circ (\op Y\square \op Y')\circ (\op Q\otimes \op Q')\;\ar [rr]^(0.4){\id\circ \big(b^{\sharp}\square(b')^{\sharp}\big)\circ \id}&&\;(\op P\otimes \op P')\circ \big((\op P\circ \op Z\circ \op Q)\square (\op P'\circ \op Z'\circ \op Q')\big)\circ (\op Q\otimes \op Q')}$$}
commutes.  The commutativity of diagram (\ref{eqn:tau}) in Theorem \ref{thm:tau}, applied to $\op Z$ and $\op Z'$, therefore suffices for us to conclude that $\Big(\xi(b, b')\xi(a, a')\Big)^{\sharp}=\Big(\xi(ba, b'a'))\Big)^{\sharp}$.
\end{proof}

\begin{proof}[Proof of Theorem \ref{thm:tot}]   Let $\op P$ and $\op Q$ be operads with multiplications denoted $\mu$, and consider a $(\op P,\op Q)$-bimodule $\op M$, with  left $\op P$-action $\lambda$ and right $\op Q$-action $\rho$ on $\op M$.   As a special case of a well known result for algebras over a monad, we know that there is a contractible coequalizer in  $\sseq$, natural in $\op M$,
\begin{equation}\label{eqn:contr-coeq}
\xymatrix{\op P^{\circ 2}\circ \op M\circ \op Q^{\circ 2} \ar @<1.5ex>[rr]^(0.55){a_{\op M}}\ar @<-1.5ex>[rr]_(0.55){b_{\op M}}&&\op P\circ \op M\circ \op Q\ar [ll] |(0.45){s_{\op M}}\ar @<1ex>[rr]^(0.6){q_{\op M}}&&\op M\ar @<1ex> [ll]^(0.4){e_{\op M}},}
\end{equation}
where $a_{\op M}=\id_{\op P}\circ\rho(\lambda \circ \id_{\op P})\circ \id_{\op P}$, $b_{\op M}=\mu\circ \id_{\op M}\circ \mu$, $s_{\op M}=\eta_{\op P}\circ \id_{\op P\circ \op M\circ \op P}\circ\eta_{\op P}$, $q_{\op M}=\rho(\lambda \circ \id_{\op P})$ and $e_{\op M}=\eta_{\op P}\circ \id_{\op M}\circ\eta_{\op P}$.  Moreover, the morphisms $a_{\op M}$, $b_{\op M}$ and $q_{\op M}$ all underlie morphisms of $(\op P,\op Q)$-bimodules such that
\begin{equation*}\label{eqn:coeq}
\xymatrix{F_{\op P, \op Q}(\op P\circ \op M\circ \op Q) \ar @<1ex>[rr]^(0.55){a_{\op M}}\ar @<-1ex>[rr]_(0.55){b_{\op M}}&&F_{\op P, \op Q}(\op M)\ar [rr]^(0.6){q_{\op M}}&&\op M}
\end{equation*}
is a coequalizer in  $\bimod{\op P,\op Q}$.

Let $\op M$ be a $(\op P,\op Q)$-bimodule and $\op N$ a $(\op P',\op Q')$-bimodule. Define $\op M\tot \op N$ to be (any representative of) the colimit of the following diagram of $(\op P\otimes \op P', \op Q\otimes \op Q')$-bimodules
\begin{equation}\label{eqn:double-coeq}
\xymatrix{F_{\op P\otimes \op P', \op Q\otimes \op Q'}\big( (\op P\circ \op M\circ\op Q) \square (\op P'\circ \op N\circ \op Q')\big)  \ar @<1ex>[rr]^(0.55){\xi(\id,a_{\op N})} \ar @<-1ex>[rr]_(0.55){\xi(\id,b_{\op N})} \ar @<1ex>[d]^{\xi(a_{\op M},\id)}\ar @<-1ex>[d]_{\xi(b_{\op M},\id)} &&F_{\op P\otimes\op P', \op Q\otimes \op Q'}( (\op P\circ \op M\circ\op Q) \square \op N) \ar @<1ex>[d]^{\xi(a_{\op M},\id)}  \ar @<-1ex>[d]_{\xi(b_{\op M},\id)}\\
F_{\op P\otimes \op P', \op Q\otimes \op Q'}\big( \op M \square (\op P'\circ \op N\circ \op Q')\big)  \ar @<1ex>[rr]^(0.55){\xi(\id,a_{\op N})} \ar @<-1ex>[rr]_(0.55){\xi(\id,b_{\op N})} &&F_{\op P\otimes\op P', \op Q\otimes \op Q'}( \op M \square \op N).}
\end{equation}
 By the universal property of colimits, for any morphism $c:\op M\to \op M'$ of $(\op P,\op Q)$-bimodules and any morphism $d:\op N\to \op N'$ of $(\op P',\op Q')$-bimodules, there is a unique, induced morphism of $(\op P\otimes\op P', \op Q\otimes \op Q')$-bimodules
$$c\tot d: \op M\tot \op N \to \op M'\tot \op N',$$
which is necessarily the identity, if $c$ and $d$ are identities.
By Proposition \ref{prop:composition} the tensor product of any bimodule morphisms is compatible with composition, by the universal property of the colimit, so that 
$$-\tot -: \bimod{\op P, \op Q}\times \bimod {\op P',\op Q'} \to \bimod {\op P\otimes \op P', \op Q\otimes \op Q'}$$ 
is indeed a functor.

It remains to check that $F_{\op P, \op Q}\op X\tot F_{\op P',\op Q'}\op X'$ can be chosen to be $\op  F_{\op P\otimes \op P',  \op Q\otimes \op Q'}(\op X\square \op X')$ for all symmetric sequences $\op X$ and $\op X'$.  Observe first that the contractible coequalizer (\ref{eqn:contr-coeq}) for $\op M=F_{\op P,\op Q}\op X$ induces a contractible coequalizer for any symmetric sequence $\op Z$
{\small\begin{equation*}\label{eqn:ugh}
\xymatrix{(\op P\otimes \op P')\circ\big((\op P^{\circ 2}\circ \op X\circ \op Q^{\circ 2})\square \op Z\big)\circ (\op Q\otimes \op Q') \ar @<1.75ex>[rr]^(0.53){\xi(a_{F_{\op P,\op Q}\op X},\id)}\ar @<-1.75ex>[rr]_(0.53){\xi(b_{F_{\op P,\op Q}X},\id)}&&(\op P\otimes \op P')\circ\big((\op P\circ \op X\circ \op Q)\square \op Z\big)\circ (\op Q\otimes \op Q')\ar [ll] |(0.47){\xi(s_{F_{\op P,\op Q}X},\id)}\ar @<1ex>[d]^{\xi(q_{F_{\op P,\op Q}X},\id)}\\
&&(\op P\otimes \op P')\circ(\op X\square \op Z)\circ (\op Q\otimes \op Q')\ar @<1ex> [u]^{\xi(e_{F_{\op P,\op Q}X},\id)}}
\end{equation*}}
in $\sseq$, where $\id$ refers to the identity morphism on $F_{\op P',\op Q'}\op Z$.  Consequently, 
{\begin{equation*}
\xymatrix{F_{\op P\otimes \op P', \op Q \otimes \op Q'}\big((\op P^{\circ 2}\circ \op X\circ \op Q^{\circ 2})\square \op Z\big) \ar @<1ex>[rr]^(0.5){\xi(a_{F_{\op P,\op Q}X},\id)}\ar @<-1ex>[rr]_(0.5){\xi(b_{F_{\op P,\op Q}X},\id)}&&F_{\op P\otimes \op P',\op Q\otimes \op Q'}\big((\op P\circ \op X\circ \op Q)\square \op Z\big)\ar [d]^{\xi(q_{F_{\op P,\op Q}X},\id)}\\
&&F_{\op P\otimes \op P', \op Q\otimes \op Q'}(\op X\square \op Z)}
\end{equation*}}
is a coequalizer in  $\bimod{\op P\otimes \op P', \op Q\otimes \op Q'}$. 

Consider diagram (\ref{eqn:double-coeq}) for $\op M=F_{\op P,\op Q}\op X$ and $\op N=F_{\op P',\op Q'}\op X'$.  Applying the observation above to the two horizontal coequalizers and then to the resulting vertical coequalizer, we see that $F_{\op P\otimes \op P',\op Q\otimes \op Q'}(\op X\square\op X')$ does indeed represent the colimit of (\ref{eqn:double-coeq}), as desired. Moreover, if $a:F_{\op P,\op Q}\op X \to F_{\op P,\op Q}\op Y$ and $b: F_{\op P',\op Q'}\op X'\to F_{\op P',\op Q'}\op Y'$ are bimodule morphisms, then the universal property of the colimit implies that $a\tot b=\xi(a,b)$.
\end{proof}

\section{The algebra of the divided powers functor}

The divided powers functor is interesting not only for the role it plays in the matrix monoidal structure.  We describe in this section certain algebraic properties of the divided powers functor, which we will apply in an forthcoming article, when we construct a model for the space of long links.

\subsection{Divided powers and bimodules}

We establish in this section that for all symmetric operads $\op P$, $\op P'$, $\op Q$ and $\op Q'$ and for all $(\op P',\op Q')$-bimodules $\op M'$, the functor $-\tot \op M': \bimod{\op P,\op Q} \to \bimod{\op P\otimes\op P', \op Q\otimes \op Q'}$  possesses a right adjoint, constructed using the divided powers functor.  We first need to show that the divided powers functor preserves bimodule structure, which is interesting in itself.

\begin{prop}\label{prop:bimod} Let $\op P$ and $\op Q$ be symmetric operads.  For every $n\geq 1$, the $n^{\text{th}}$-divided powers functor $\gamma^{\mathfrak S}_{n}:\sseq\to \sseq$ restricts and corestricts to a functor $$\gamma^{\mathfrak S}_{n}: \bimod{(\op P, \op Q)} \to \bimod{(\op P, \op Q)}.$$
\end{prop}

\begin{proof} Let $\lambda : \op P\circ \op M\to \op M$ and $\rho: \op M\circ \op Q\to \op M$ denote the left $\op P$-action and the right $\op Q$-action on a $(\op P, \op Q)$-bimodule $\op M$.  For any $x\in \op M(mn)$, let $\gamma_{n}^{\mathfrak S}(x)$ denote the corresponding element of $\gamma_{n}^{\mathfrak S}(\op M)(m)$.

Define $\widetilde\lambda: \op P\circ \gamma_{n}^{\mathfrak S}(\op M)\to \gamma_{n}^{\mathfrak S}(\op M)$ and $\widetilde \rho: \gamma_{n}^{\mathfrak S}(\op M) \circ \op Q\to \gamma_{n}^{\mathfrak S}(\op M)$ by
$$\widetilde\lambda \big(p; \gamma_{n}^{\mathfrak S}(x_{1}),...,\gamma_{n}^{\mathfrak S}(x_{k})\big)=\gamma_{n}^{\mathfrak S}\big(\lambda (p;x_{1},...,x_{k})\big)$$
and
$$\widetilde\rho\big(\gamma_{n}^{\mathfrak S}(x); q_{1},...,q_{l}\big)=\gamma_{n}^{\mathfrak S}\big(\rho(x; \underbrace{q_{1},..,q_{1}}_{n},....,\underbrace{q_{l},..,q_{l}}_{n})\big)$$
for all $p\in \op P(k)$, $x_{i}\in \op M(m_{i}n)$ for $1\leq i\leq k$, $x\in \op M(ln)$ and $q_{j}\in \op Q(r_{j})$ for $1\leq j\leq l$.

It is a straightforward exercise to check that $\widetilde\lambda$ and $\widetilde\rho$ endow with $\gamma_{n}^{\mathfrak S}(\op M)$ with the structure of a $(\op P, \op Q)$-bimodule.  One needs only to be a little careful in establishing equivariance with respect to the symmetric group actions on the right and in showing that $\widetilde \rho$ is well defined.  It is helpful to remark that for any $k,l\geq1$, the homomorphism $\nu_{k,l}(-,e): \mathfrak S_{k}\to \mathfrak S_{kl}$ factors, via the diagonal map, through the homomorphism $\alpha:\mathfrak S_{k}^{\times l}\to \mathfrak S_{kl}$ given by 
$$\alpha(\vp_{1},...,\vp_{l})(i,j)=\big(\vp_{j}(i),j\big),$$
where we see $\mathfrak S_{kl}$ as the set of permutations of $\{1,...,k\}\times\{1,...,l\}$.
\end{proof}

Before proving the main result of this section, we need a bit more notation.

\begin{notn} For any symmetric operads $\op P$ and $\op Q$, and any $(\op P, \op Q)$-bimodules $\op M$ and $\op N$, let $\map_{\op P, \op Q}(\op M, \op N)$ denote the simplicial mapping space of $(\op P, \op Q)$-bimodule maps from $\op M$ to $\op N$, which can be constructed as the equalizer of the obvious two maps from $\map _{\mathfrak S}(\op M, \op N)$ to $\map _{\mathfrak S} (\op P\circ \op M\circ \op Q, \op N) $, built from the $\op P$- and $\op Q$-actions on $\op M$ and $\op N$. 
\end{notn}

\begin{prop}\label{prop:closed} Let $\op P$, $\op P'$, $\op Q$ and $\op Q'$ be symmetric operads. For every $(\op P',\op Q')$-bimodule $\op M'$, there is an adjunction
$$\adjunct{\bimod{\op P, \op Q}}{\bimod{\op P\otimes\op P',  \op Q\otimes \op Q'}}{-\tot \op M'}{\map_{\op P',\op Q'}\big(\op M', \gamma^{\mathfrak S}_{\bullet}(-)\big)}.$$
\end{prop}  

\begin{proof} Let $\op M$ be any $(\op P, \op Q)$-bimodule, and consider the coequalizer
\begin{equation*}\label{eqn:coeq2}
\xymatrix{F_{\op P, \op Q}(\op P\circ \op M\circ \op Q) \ar @<1ex>[rr]^(0.55){a_{\op M}}\ar @<-1ex>[rr]_(0.55){b_{\op M}}&&F_{\op P, \op Q}(\op M)\ar [rr]^(0.6){q_{\op M}}&&\op M}
\end{equation*}
in  $\bimod{\op P,\op Q}$.  Let $\op X'$ be any symmetric sequence. Since 
$$F_{\op P, \op Q}(\op X)\tot F_{\op P', \op Q'}(\op X') =F_{\op P\otimes \op P', \op Q\otimes \op Q'}(\op X\square \op X')$$ 
for every symmetric sequence $\op X$, it follows easily, by a comparison of colimits similar that at the end of the proof of Theorem \ref{thm:tot}, that 
{\small{\begin{equation}\label{eqn:coeq3}
\xymatrix{F_{\op P, \op Q}(\op P\circ \op M\circ \op Q)\tot F_{\op P', \op Q'}(\op X')  \ar @<1ex>[rr]^(0.55){a_{\op M}\tot \id}\ar @<-1ex>[rr]_(0.55){b_{\op M}\tot \id}&&F_{\op P, \op Q}(\op M)\tot F_{\op P', \op Q'}(\op X')\ar [r]^(0.55){q_{\op M}\tot \id}&\op M\tot F_{\op P', \op Q'}(\op X')}
\end{equation}}}

\noindent is a coequalizer diagram in $\bimod{\op P\otimes \op P',\op Q\otimes \op Q'}$.   Another straightforward comparison of colimits then implies that 
{\small{\begin{equation}\label{eqn:coeq4}
\xymatrix{\op M\tot F_{\op P', \op Q'}(\op P'\circ \op M'\circ \op Q')  \ar @<1ex>[rr]^(0.55){\id\tot a_{\op M'}}\ar @<-1ex>[rr]_(0.55){\id\tot b_{\op M'}}&&\op M\tot F_{\op P', \op Q'}(\op M')\ar [rr]^(0.6){\id \tot q_{\op M'}}&&\op M\tot \op M'}
\end{equation}}}
is also a coequalizer diagram in $\bimod{\op P\otimes \op P',\op Q\otimes \op Q'}$, for any $(\op P', \op Q')$-bimodule $\op M'$.

We check the adjunction first in the case of free bimodules. Let $\op X$ and $\op X'$ be symmetric sequences and $\op N$ any $(\op P\otimes \op P', \op Q\otimes \op Q')$-bimodule.
{\small\begin{align*}
\bimod{\op P\otimes  \op P',\op Q\otimes \op Q'}\big(F_{\op P,\op Q}\op X \tot F_{\op P',\op Q'}\op X', \op N\big)&=\bimod{\op P\otimes  \op P',\op Q\otimes \op Q'}\big(F_{\op P\otimes \op P', \op Q\otimes \op Q'}(\op X\square\op X'), \op N\big)\\
&\cong\sseq(\op X\square \op X', \op N)\\
&\cong \sseq \bigg(\op X, \map_{\mathfrak S}\big(\op X', \gamma_{\bullet}^{\mathfrak S}(\op N)\big)\bigg)\\
&\cong \sseq\bigg(\op X, \map_{\op P', \op Q'}\big(F_{\op P',\op Q'}\op X', \gamma_{\bullet}^{\mathfrak S}(\op N)\big)\bigg)\\
&\cong \bimod{\op P,\op Q}\bigg(F_{\op P, \op Q}\op X, \map_{\op P',\op Q'}\big(F_{\op P',\op Q'}\op X', \gamma_{\bullet}^{\mathfrak S}(\op N)\big)\bigg),\\
\end{align*}}
\noindent where we applied Proposition \ref{prop:adjunction} in the third step. Note that we also used above that $\gamma_{\bullet}^{\mathfrak S}(\op N)\in (\bimod{\op P\otimes \op P', \op Q\otimes \op Q'})^{\mathbb N}$, so that 
$\map_{\op P',\op Q'}\big(F_{\op P',\op Q'}\op X', \gamma_{\bullet}^{\mathfrak S}(\op N)\big)$ is naturally a $(\op P,\op Q)$-bimodule.

Together with coequalizer (\ref{eqn:coeq3}), the isomorphism above implies that
$$\bimod{\op P\otimes  \op P',\op Q\otimes \op Q'}\big(\op M \tot F_{\op P',\op Q'}\op X', \op N\big)\cong\bimod{\op P, \op Q}\bigg(\op M, \map_{\op P',\op Q'}\big(F_{\op P',\op Q'}\op X', \gamma_{\bullet}^{\mathfrak S}(\op N)\big)\bigg)$$
for every $(\op P,\op Q)$-bimodule $\op M$, symmetric sequence $\op X'$ and $(\op P\otimes \op P', \op Q\otimes \op Q')$-bimodule $\op N$.  We can therefore conclude from coequalizer (\ref{eqn:coeq4}) that  
$$\bimod{\op P\otimes  \op P',\op Q\otimes \op Q'}\big(\op M \tot \op M', \op N\big)\cong\bimod{\op P, \op Q}\bigg(\op M, \map_{\op P', \op Q'}\big(\op M', \gamma_{\bullet}^{\mathfrak S}(\op N)\big)\bigg)$$
for every $(\op P,\op Q)$-bimodule $\op M$, $(\op P', \op Q')$-bimodule $\op M'$ and $(\op P\otimes \op P', \op Q\otimes \op Q')$-bimodule $\op N$.
\end{proof}

\subsection{Divided powers and operads}

Under nice conditions, the divided powers functor applied to a nonsymmetric operad gives rise to a nonunital operad, which can have interesting applications, as we will show in our construction of a model for the space of long links.  

Let $\cat{Op^{nu}}$ denote the category of nonunital, nonsymmetric operads, i.e., of objects $\op P\in \seq{}$, endowed with an associative multiplication $\mu: \op P\circ \op P\to \op P$, though not necessarily any unit.

\begin{defn}\label{defn:axial} An \emph{$n$-axial structure} on a nonsymmetric operad $(\op P, \mu)$ consists of a set
$$\big\{\iota_{k}: \op P(kn) \to \op P(k)^{\times n}\mid k\geq 0\big\}$$
of simplicial monomorphisms such that for all $k\geq 1$ and $m_{1},.., m_{k}\in \mathbb N$, the composite 
{\small $$\xymatrix{\op P(kn)\times \op P (m_{1}n)\times \cdots \times \op P(m_{k}n) \ar [d]_{\iota _{k}\times \iota_{m_{1}}\times\cdots \times \iota _{m_{k}}}&&\\
\op P(k)^{\times n}\times \op P (m_{1})^{\times n}\times \cdots \times \op P(m_{k})^{\times n}\ar[r]^{\cong}&\big(\op P(k)\times \op P (m_{1})\times \cdots \times \op P(m_{k})\big)^{\times n}\ar [r]^(0.7){\mu^{\times n}}&\op P(m)^{\times n}}$$}

\noindent factors through $\iota _{m}$, i.e., there exists $\widehat\mu: \op P(kn)\times \op P (m_{1}n)\times \cdots \times \op P(m_{k}n) \to \op P(mn)$ such that
{\small $$\xymatrix{\op P(kn)\times \op P (m_{1}n)\times \cdots \times \op P(m_{k}n) \ar [d]_{\iota _{k}\times \iota_{m_{1}}\times\cdots \times \iota _{m_{k}}}\ar [rr]^{\widehat\mu}&&\op P(mn)\ar[d]^{\iota_{m}}\\
\op P(k)^{\times n}\times \op P (m_{1})^{\times n}\times \cdots \times \op P(m_{k})^{\times n}\ar[r]^{\cong}&\big(\op P(k)\times \op P (m_{1})\times \cdots \times \op P(m_{k})\big)^{\times n}\ar [r]^(0.7){\mu^{\times n}}&\op P(m)^{\times n}}$$}

\noindent commutes.
\end{defn}

\begin{rmk} Since the maps $\iota_{k}$ comprising an $n$-axial structure are monomorphisms, the morphism $\widehat \mu$ of the definition above is necessarily unique.  It follows that any morphism of operads with $n$-axial structure that preserves the axial structure must also commute with the corresponding morphisms $\widehat\mu$.
\end{rmk}

\begin{rmk} Igusa proved in \cite{igusa} that the forgetful functor $U$ from reduced operads to simplicial monoids, which sends an operad to its component in arity 1, admits a right adjoint $R$ such that $RM(k)=M^{\times k}$, for any simplicial monoid $M$.  In \cite[Definition 7.6] {fiedorowicz-vogt}, Fiedorowicz and Vogt defined an \emph{axial operad} in $\cat{sSet}$ to be a reduced, $\mathfrak S$-free operad $\op P$ such that the unit map $\eta_{\op P}:\op P\to RU\op P$, which is a morphism of operads, is arity-wise injective, i.e., $\eta_{\op P}(k):\op P(k)\to \op P(1)^{\times k}$ is an injection for every $k$.  Every axial operad clearly admits an $n$-axial structure for all $n\geq 1$.
\end{rmk}

\begin{notn} Let $\cat{Op^{ax_{n}}}$ denote the category of operads endowed with $n$-axial structures and of morphisms preserving this extra structure.
\end{notn}

\begin{ex} The nonsymmetric associative operad $\op A$ admits an $n$-axial structure for all $n$, since $\op A(k)$ is  a constant simplicial set on exactly one point.  If $\delta_{k}$ denotes the unique element of $\op A(k)_{0}$, then we can define $\iota_{kn}(\delta_{kn})=(\delta_{k},....,\delta_{k})$. 
\end{ex}

\begin{ex}  As Fiedorowicz and Vogt showed in \cite[Section 8]{fiedorowicz-vogt}, the reduced $W$-construction $W\op P$ on any simplicial operad $\op P$ is naturally an axial operad, so that every simplicial operad admits an axial cofibrant replacement.   They proved furthermore that if $\op P$ and $\op Q$ are the nerve of the operads in $\cat {Cat}$ parametrizing $m$-fold and $n$-fold monoidal categories, respectively, then $W\op P\otimes W\op Q$ is also axial  \cite[Corollary 10.6]{fiedorowicz-vogt}.  This result plays a crucial role in their proof that the Boardman-Vogt tensor product of a cofibrant $E_{m}$-operad and a cofibrant $E_{n}$-operad is an $E_{m+n}$-operad.
\end{ex}

\begin{ex} For all $m$, the nonsymmetric operad underlying the little $m$-balls operad $\op B_{m}$ admits an $n$-axial structure for all $n$.  Let $B^{m}$ denote the open unit $m$-ball in $\mathbb R^{m}$, centered at the origin. Recall that $\op B_{m}(k)=\operatorname{sEmb}(\coprod_{k}B^{m}, B^{m})$, the space of standard embeddings, i.e., of embeddings $g:\coprod_{k}B^{m}\to B^{m}$ given on each component by translation and multiplication by a positive scalar \cite{arone-turchin}.  We denote  an element of $\op B_{m}(k)$ by $(b_{1},...,b_{k})$, an ordered sequence of standard embeddings.

 For all $k\geq 1$, let 
$$\iota_{k}:\op B_{m}(kn)\to \op B_{m}(k)^{\times n}: (b_{1},..., b_{kn}) \mapsto \big( (b_{1,1},...,b_{k,1}),...,(b_{1,n},..., b_{k,n})\big),$$
where $b_{i,j}:=b_{(i-1)n+j}$ for all $1\leq i\leq k$ and $1\leq j\leq n$, which is clearly injective.  The morphism $\iota _{k}$ thus partitions a collection of $kn$ disjoint balls contained in $B^{m}$ into $n$ collections of $k$ disjoint balls.
\medskip

\centerline{\includegraphics [scale=0.8]{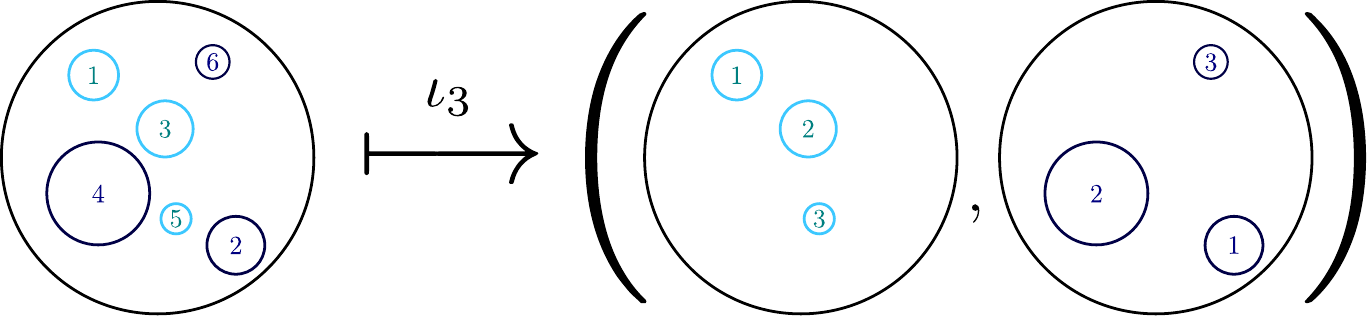}}
\medskip

Moreover it is not hard to see that we can define an appropriate
$$\widehat \mu:\op B_{m}(kn)\times \op B_{m}(l_{1}n) \times \cdots \times \op B_{m}(l_{k}n) \to \op B_{m}(ln)$$
by
$$\widehat \mu \big( (b_{1},..., b_{kn}); (b^{1}_{1},..., b^{1}_{l_{1}n}),..., (b^{k}_{1},..., b^{k}_{l_{k}n})\big)=(b'_{1},...,b'_{ln}),$$
where
$$(b'_{j}, b'_{n+j},...,b'_{(l-1)n+j})=(b_{1,j},...,b_{k,j})\big((b^{1}_{1,j_{}},...,b^{1}_{l_{1},j_{}}),...,(b^{k}_{1,j_{}},...,b^{k}_{l_{k},j_{}})\big)$$
for all $1\leq j\leq n$. 
\bigskip

\centerline{\includegraphics [scale=0.8]{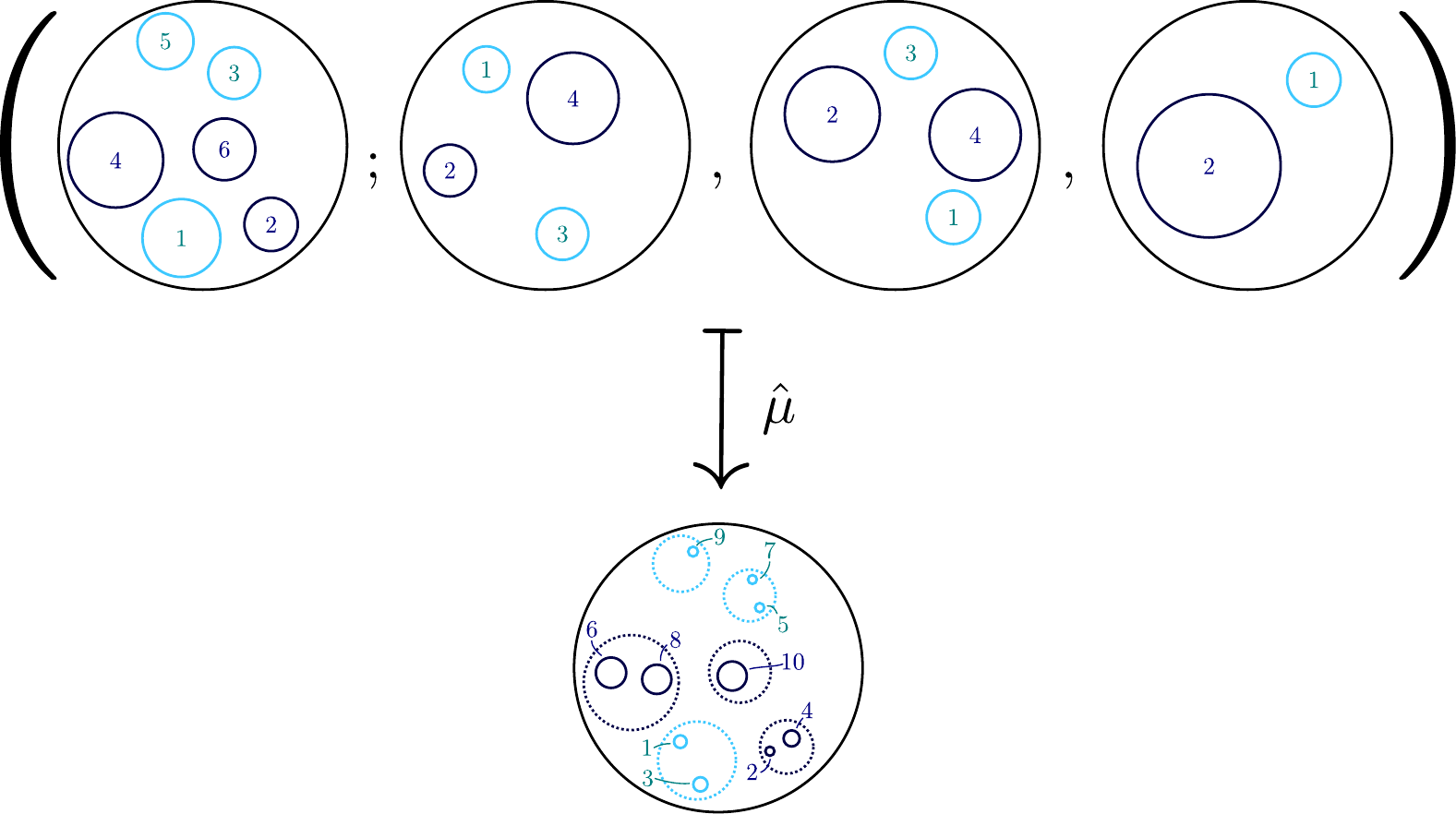}}
\bigskip

In other words, given collections of $kn$ disjoint balls and of $l_{s}n$ disjoint balls for $1\leq s\leq k$, the map $\widehat \mu$ evaluates the $j^{\text{th}}$ collection of $k$ disjoint balls, for each $j$,   on the sequence of $j^{\text{th}}$ collections of $l_{s}$ disjoint balls, for $1\leq s\leq k$, where the partitions used are those given by $\iota_{k}$ and $\iota _{ls}$, $1\leq s\leq k$.

The natural operad map $\vp_{m}:\op B_{1}\to \op B_{m}$, induced by the inclusion $\R \hookrightarrow \R^{m}$, clearly respects $n$-axial structure. 

\end{ex}

\begin{prop}\label{prop:axial}The nonsymmetric divided powers functor induces a functor
$$\gamma_{n}:\cat{Op^{ax_{n}}}\to \cat{Op^{nu}}.$$
\end{prop}

\begin{proof} For any $x\in \op P(mn)$ and any $m$, let $\gamma_{n}(x)$ denote the corresponding element of $\gamma_{n}(\op P)(m)$. Using the notational conventions of Definition \ref{defn:axial}, define $\widetilde\mu: \gamma_{n}(\op P)\circ \gamma_{n}(\op P) \to \gamma_{n}(\op P)$ by
$$\widetilde \mu\big(\gamma_{n}(p); \gamma_{n}(p_{1}),..., \gamma_{n}(p_{k})\big)=\gamma_{n}\big(\widehat \mu(p;p_{1},...,p_{k})\big)$$
for all $p\in \op P(kn)$ and $p_{i}\in \op P(m_{i}n)$ for $1\leq i\leq k$.   The associativity of $\mu$ and the fact that each $\iota_{k}$ is a monomorphism together imply that $\widetilde\mu$ is associative and therefore endows $\gamma_{n}(\op P)$ with the structure of a nonunital operad.  The naturality of this construction with respect to morphisms in $\cat{Op^{ax_{n}}}$ is obvious.
\end{proof}

\begin{ex}\label{ex:gammaphi} Since the canonical operad morphism $\vp_{m}:\op B_{1}\to \op B_{m}$ preserves $n$-axial structures for all $n$, there are induced morphisms of nonunital operads $\gamma_{n}\vp_{m}:\gamma_{n}\op B_{1}\to \gamma_{n}\op B_{m}$, which also respect the induced symmetric structures.
\end{ex}

\begin{ex} \label{ex:diagonal}  For all $n\geq 1$, there is a curious but useful morphism of symmetric, nonunital operads 
$$\Delta_{n}:\op B_{1}\to \gamma_{n}\op B_{1}$$
defined as follows.   Let $(b_{1},...,b_{k})$ be an ordered sequence of  pairwise disjoint, open intervals contained in $]-1,1[$, with $b_{i}=]x_{i},y_{i}[$. Set
$$\Delta_{n}(b_{1},...,b_{k})=(b_{1,1},..., b_{1,n},...,b_{k,1},...,b_{k,n}),$$
where 
$$b_{i,j}=\Big]\frac 1n (x_{i}+2j-n-1), \frac 1n(y_{i}+2j-n-1)\Big[,$$
for all $1\leq i\leq k$ and $1\leq j\leq n$.  The effect of applying $\Delta_{n}$ is shrink the $k$ intervals $b_{1}$,...,$b_{k}$ by a factor of $n$, then to embed a copy of the shrunken intervals into each of the open subintervals $]\frac{2j-n-1}n, \frac {2j-n}n[$ of $]-1,1[$, for all $1\leq j\leq n$.  It is an easy exercise to check that $\Delta_{n}$ is a morphism of nonunital operads.
\end{ex}

\begin{rmk}  If $\op P$ is an operad with unit element $x_{1}\in \op P$, then $\gamma_{n}(\op P)$ is a unital operad if and only if there is an element $x_{n}\in \op P(n)$ such that $\iota_{1} (x_{n})=(x_{1},...,x_{1})$.  For example, this is not the case if $\op P=\op B_{m}$, the little $m$-balls operad, but it is the case if $\op P=\op A$, the nonsymmetric associative operad.  Indeed, $\gamma_{n}(\op A)\cong \op A$ as unital operads.
\end{rmk}

\subsection{Divided powers and various monoidal structures}
Besides its matrix monoidal structure, the category $\sseq$ admits at least two other closed, symmetric monoidal structures: the \emph{levelwise} structure and the \emph{graded} structure, with respect to both of which the divided powers functor is monoidal, as we show here.   

Let $\op X$ and $\op Y$ be symmetric sequences. Their levelwise monoidal product, $\op X\times \op Y$, is the categorical product, specified by
$$(\op X\times\op Y)(n)=\op X(n)\times \op Y(n),$$
endowed with the diagonal $\mathfrak S_{n}$-action. Their graded monoidal product, $\op X\odot \op Y$, which is given by Day convolution of the monoidal structure on $\mathfrak S$ that is additive on objects, satisfies
$$(\op X\odot \op Y)(n)= \coprod_{l+m=n}\op X(l)\times \op Y(m) \times_{\mathfrak S _{l}\times \mathfrak S _{m}}\mathfrak S_{n},$$
where $\mathfrak S _{l}\times \mathfrak S_{m}$ is seen as a subgroup of $\mathfrak S _{n}$ in the usual way: $\mathfrak S _{l}$ permutes $\{1,..,l\}$ and $\mathfrak S_{m}$ permutes $\{l+1,...,n\}$.

\begin{prop}\label{prop:grmon} For every $n\geq 1$, the divided powers functor $\gamma_{n}^{\mathfrak S}:\sseq\to \sseq$ is 
\begin{enumerate}
\item strongly monoidal with respect to the levelwise monoidal structure, and
\item monoidal with respect to the graded monoidal structure.
\end{enumerate}
\end{prop}

\begin{rmk} From this proposition it follows that $\gamma_{n}^{\mathfrak S}$ preserves both levelwise and graded monoids.
\end{rmk}

\begin{proof} (1) It is very easy to see that $\gamma_{n}^{\mathfrak S}(\op X\times \op Y)$ is naturally isomorphic to $\gamma_{n}^{\mathfrak S}(\op X)\times \gamma_{n}^{\mathfrak S}(\op Y)$.

(2) If $\op X$ and $\op Y$ are symmetric sequences, then for all $k\geq 0$,
$$\big(\gamma_{n}^{\mathfrak S}(\op X)\odot \gamma_{n}^{\mathfrak S}(\op Y)\big)(k)=\coprod_{i+j=k}\nu_{i,n}^{*}\op X(in)\times \nu_{j,n}^{*}\op Y(jn)\times _{\mathfrak S_{i}\times \mathfrak S_{j}} \mathfrak S_{k},$$
while
$$\gamma_{n}^{\mathfrak S}(\op X \odot \op Y)(k)=\nu_{k,n}^{*}\big(\coprod_{l+m=kn} \op X(l)\times \op Y(m) \times_{\mathfrak S _{l}\times \mathfrak S_{m}} \mathfrak S_{kn}\big).$$
There is therefore a natural morphism of symmetric sequences
$$\iota:\gamma_{n}^{\mathfrak S}(\op X)\odot \gamma_{n}^{\mathfrak S}(\op Y) \to \gamma_{n}^{\mathfrak S}(\op X \odot \op Y),$$
given simply by inclusion, which can easily be seen to be appropriately equivariant, associative and unital.
\end{proof}

\begin{rmk}  The results in this section can be generalized to totally disconnected groupoids $\G$ with set of objects $\N$ that are endowed with two monoidal structures, one of which is multiplicative and the other additive on objets, and which satisfy a sufficiently strict interchange law.
\end{rmk}

\appendix

\section{Proofs of Proposition \ref{prop:sigma} and Theorem \ref{thm:tau}}

Before proving the results cited above, we introduce an alternative, coordinate-free approach to symmetric sequences and their monoidal structures, to simplify and clarify our arguments.  We are not aware of exactly this approach to monoidal structures on symmetric sequences elsewhere in the literature and think it may prove useful in other contexts as well, as it enables us to avoid explicit mention of symmetric group actions.

\subsection {Coordinate-free operads}

Let $\F$ denote the category of finite sets and all set maps between them. Let $\A$ denote the groupoid of finite set maps, i.e., the objects of $\A$ are morphisms in $\F$, while a morphism in $\A$ from an object $f:T\to S$ to an object $f':T'\to S'$ consists of a pair of bijections $\beta: T\to T'$ and $\alpha: S\to S'$ such that $f'\beta=\alpha f$.  Note that there is an injective homomorphism of groupoids
$$\Sigma \to \A: n \mapsto \big(\{1,...,n\} \to \{1\}\big), \sigma \in \Sigma_{n} \mapsto (\sigma, \id_{\{1\}}),$$
enabling us to view $\Sigma$ as a subgroupoid of $\A$.

The category $\A^{op}$ is symmetric monoidal with respect to $\coprod$, the restriction of the coproduct functor in the category of finite set maps and all pairs of set maps between them to $\A^{op}$. Note that $\coprod$ is not the coproduct in $\A^{op}$ itself:  given maps of finite sets $f$ and $g$,  there are no maps in $\A^{op}$ from $f$ and $g$ to $f\coprod g$.   

A \emph{multiplicative functor} from $\A^{op}$ to $\cat S$ is a functor $\Phi:\A^{op}\to \cat S$ that is monoidal with respect to $\coprod$ on $\A^{op}$ and with respect to cartesian product on $\cat S$.  In particular, if $\Phi$ is a multiplicative functor, then the natural isomorphism $$ \Phi ( -\coprod-) \xrightarrow\cong \Phi (-)\times \Phi (-)$$ induces a natural isomorphism
\begin{equation}\label{eqn:product-formula}
\Phi(f:S\to R) \cong \prod_{r\in R} \Phi \big( f^{-1}(r) \to \{ r\}\big)
\end{equation}
for every map of finite sets $f:S\to R$. 

We let $[\cat A^{{op}}, \S]_{mult}$ denote the category of multiplicative functors and monoidal natural transformations, i.e.,  natural transformations $\tau: \Phi \to \Phi'$ such that 
$$\xymatrix{\Phi (f\coprod g) \ar [d]_{\tau_{f\coprod g}}\ar [r]^(0.4){\cong } & \Phi (f) \times \Phi (g)\ar [d]^{\tau_{f}\times \tau _{g}}\\
\Phi' (f\coprod g)\ar [r]^(0.4){\cong } & \Phi '(f) \times \Phi' (g)
}$$
commutes for all $f,g\in \A$. 

Our interest in the category $[\cat A^{{op}}, \S]_{mult}$ is due to the key lemma below.

\begin{lem}\label{lem:sseq} The category $\sseq$ of symmetric sequences is equivalent to the category $[\cat A^{{op}}, \S]_{mult}$.
\end{lem}

\begin{proof} Define functors
$$\sseq \to [\cat A^{{op}}, \S]_{mult}: \op X \mapsto \Phi_{\op X}$$
and
$$[\cat A^{{op}}, \S]_{mult}\to \sseq: \Phi\mapsto \op X_{\Phi}$$
as follows.  For all $\Phi\in [\A^{{op}}, \S]_{mult}$, the symmetric sequence $\op X_{\Phi}$ is given by restriction of $\Phi$ to $\Sigma^{op}$, seen as a subcategory of $\A^{op}$.

To define the other equivalence, first fix an equivalence of categories 
$$\F \to \operatorname{Sk}(\F): (T\xrightarrow{f} S) \mapsto \big(|T| \xrightarrow {|f|} |S|\big)$$
from $\F$ to a skeleton of $\F$.  This equivalence restricts and corestricts to an equivalence of groupoids between $\F_{bij}$, the category of finite sets and bijections, and $\Sigma$. For any symmetric sequence $\op X$, define $\Phi_{\op X}:\A^{{op}}\to\S$ by
$$\Phi_{\op X} (T\xrightarrow f S)=\prod_{s\in S}\op X\big(|f^{-1}(s) |\big),$$
while for any morphism $(\beta, \alpha): (T\xrightarrow f S) \to (T'\xrightarrow {f'} S')$ in $\A$,
$$\Phi_{\op X}(\beta, \alpha):\prod_{s'\in S'}\op X\Big(|(f')^{-1}(s') |\Big)\to \prod_{s\in S}\op X\Big(|f^{-1}(s) |\Big): (x_{s'})_{s'\in S'}\mapsto \big(x_{\alpha(s)}\cdot |\beta_{s}|\big)_{s\in S},$$
where $\beta_{s}=\beta|_{f^{-1}(s)}:f^{-1}(s)\xrightarrow\cong (f')^{-1}\big(\alpha (s)\big)$. The construction of $\Phi_{\op X}$ is obviously natural in $\op X$. Moreover, the universal property of the categorical product enables us to endow this functor with a natural isomorphism $\Phi_{\op X}(-\coprod -)\xrightarrow \cong \Phi _{\op X}(-)\times \Phi_{\op X}(-)$.

To conclude that these functors form an equivalence of categories, it suffices to observe that, by formula (\ref{eqn:product-formula}), every $\Phi\in [\A^{{op}}, \S]_{mult}$ is determined by its values on set maps of the form $\{1,..,n\}\to \{1\}$.
\end{proof}

We need descriptions of the various monoidal structures on the category of symmetric sequences in terms of this more unfamiliar presentation of the category.

\begin{notn} Let $\cat {Cospan}$ denote the category of $\A^{op}\times\A^{op}$,  the objects of which are cospans, i.e., pairs of set maps $f:T\to S$ and $f':T'\to S$.

Let $\cat{Comp}$ denote the category of $\A^{op}\times\A^{op}$,  the objects of which are composable pairs of set maps $f:S\to R$ and $g:T\to S$.

For all $\Phi, \Psi\in [\A^{{op}}, \S]_{mult}$, let $\Phi\times \Psi: \A^{op}\times\A^{op}\to \S$ denote the external product, i.e., 
$$(\Phi\times \Psi) ( T\xrightarrow f S, T' \xrightarrow{f'} S)= \Phi(f)\times \Psi(f').$$ 
We use the same notation to denote the restriction of this functor to either $\cat{Cospan}$ or $\cat {Comp}$.
\end{notn}

The next two lemmas can be easily proved by explicit computation of left Kan extensions. One must be careful to check that the functors obtained do indeed admit multiplicative structure. The key observation is that if $F:\cat I \to \S$ and $G:\cat J \to \S$ are two functors from small categories into $\S$, then 
$$\operatorname{colim}_{\cat I\times \cat J} (F\times G) \cong \operatorname{colim}_{\cat I}F\times \operatorname{colim}_{\cat J}G,$$
which follows from the fact that colimits in $\S$ are stable under pullback \cite[Section 2.9]{borceux}.

\begin{lem}\label{lem:monprods} Let $\Phi, \Psi\in [\A^{{op}}, \S]_{mult}$.
Under the equivalence of $\sseq$ and $[\cat A^{{op}}, \S]_{mult}$ of Lemma \ref{lem:sseq},
\begin{enumerate}
\item the left Kan extension of $\Phi\times \Psi$ along the disjoint union functor
$$ \cat {Cospan}\to \A^{op}:( T\xrightarrow f S, T' \xrightarrow{f'} S)\mapsto (T\coprod T' \xrightarrow{f+f'} S)$$
corresponds to the graded tensor product of symmetric sequences, and
\smallskip

\item the left Kan extension of $\Phi\times \Psi$ along the fibered product functor
$$  \cat {Cospan}\to \A^{op}:( T\xrightarrow f S, T' \xrightarrow{f'} S)\mapsto (T\times_{S} T' \xrightarrow{f\times_{S}f'} S)$$
corresponds to the matrix monoidal product of symmetric sequences.
\end{enumerate}
\end{lem}

Associativity and symmetry of these two monoidal structures on symmetric sequences therefore follow immediately from the associativity and symmetry of the disjoint union and the fiber product constructions.  We have a similar description of the nonsymmetric, but still associative, composition product of symmetric sequences, which, naturally enough, arises from composition of set maps.

\begin{lem}\label{lem:comp-prod} Let $\Phi,\Psi\in [\A^{{op}}, \S]_{mult}$.  
Under the identification of $\sseq$ and $[\cat A^{{op}}, \S]_{mult}$ of Lemma \ref{lem:sseq},
 the left Kan extension of $\Phi\times \Psi$ along the composition functor
$$ \operatorname{comp}: \cat {Comp}\to \A^{op}:( S\xrightarrow f R, T \xrightarrow{g} S))\mapsto (T \xrightarrow{gf} R)$$
corresponds to the composition product of symmetric sequences.
\end{lem}

\begin{rmk}\label{rmk:operad-cf} Let $\Upsilon: \A^{op}\to \S$ denote the functor specified by $\Upsilon (f)=*$ if $f$ is a bijection and $\Upsilon (f)=\emptyset $ else.  As a consequence of Lemma \ref{lem:comp-prod}, an operad corresponds under the equivalence of Lemma \ref{lem:sseq} to a functor $\Phi\in [\A^{{op}}, \S]_{mult}$ together with natural transformations
$$\mu: \Phi \times \Phi  \to  \Phi  \circ \operatorname{comp}: \cat{Comp} \to \S$$
and 
$$\eta: \Upsilon \to \Phi,$$
satisfying the obvious associativity, unitality and naturality conditions.
\end{rmk}

\subsection{Applying coordinate-free methods}

Thanks to Lemmas \ref{lem:monprods} and \ref{lem:comp-prod}, we can now describe the source and target of the natural transformation required in Proposition \ref{prop:sigma} as left Kan extensions.    The description is formulated in terms of the following pair of categories and a certain functor between them.

Let $\cat{Cospan}(\cat {Comp})$ denote the subcategory of $\cat {Comp}\times \cat {Comp}$ the objects of which are pairs of pairs $\big( (f,g), (f',g')\big)$ such that $f$ and $f'$ have the same codomain.  Let $\cat{Comp}(\cat{Cospan})$ denote the subcategory of $\cat {Cospan}\times \cat {Cospan}$ the objects of which are pairs of pairs $\big( (f,f'), (g,g')\big)$ such that  $f:S \to R$, $f': S'\to R$, $g: T\to S\times _{R}S'$ and $g': T'\to S\times _{R}S'$, whence the composite
$$T\times _{S\times_{R}S'} T' =T\times _{R}T'\xrightarrow {g\times _{R}g'} S\times_{R} S'\xrightarrow {f\times _{R}f'} R$$
makes sense.  These two categories are linked by a functor 
$$\omega:\cat {Cospan}(\cat {Comp}) \to \cat {Comp}(\cat {Cospan})$$ 
defined on objects by
$$\omega \big( (S\xrightarrow f R,T\xrightarrow g S), (S'\xrightarrow{f'} R, T'\xrightarrow {g'} S')\big) = \big( (f,f'), (g\times _{R}{S'}, {S}\times_{R}g')\big).$$
Observe that $\omega$ is associative, in the sense that
\begin{align*}
\omega(\id \times \omega)&\big( (f,g), (f',g'), (f'', g'')\big)\\
 &=\Big((f,f',f''), (g\times _{R}S'\times _{R'}S'', S\times _{R}g'\times _{R'}S'', \times _{R}S'\times _{R'}g'')\Big)\\
&=\omega(\omega \times \id)\big( (f,g), (f',g'), (f'', g'')\big).
\end{align*}
Let $\omega ^{(2)}=\omega(\id \times \omega)=\omega(\omega \times \id).$
\smallskip

\begin{cor}\label{cor:comp-cospan}  Let $\Phi,\Phi', \Psi, \Psi'\in [\A^{{op}}, \S]_{mult}$. Under the identification of $\sseq$ and $[\cat A^{{op}}, \S]_{mult}$ of Lemma \ref{lem:sseq},
\begin{enumerate}
\item   the left Kan extension of 
{\small $$\cat{Comp}(\cat {Cospan}) \xrightarrow{\Phi\times \Phi'\times \Psi \times \Psi'} \S:\big ((f,f'),(g,g')\big)\mapsto \Phi (f)\times \Phi' (f')\times \Psi (g)\times \Psi' (g')$$}
along the functor
$$ \cat{Comp}(\cat{Cospan})\xrightarrow\vp \A^{op}: \big((f,f'), (g,g')\big) \mapsto (g\times _{R}g')\circ (f\times_{R}f'),$$
where $f:S \to R$, $f': S'\to R$, $g: T\to S\times _{R}S'$ and $g': T'\to S\times _{R}S'$, corresponds to $(\op X_{\Phi}\square \op X_{\Phi'})\circ (\op X_{\Psi}\square \op X_{\Psi'})$, and
\medskip

\item  the left Kan extension of 
{\small $$ \cat{Cospan}(\cat {Comp}) \xrightarrow{\Phi\times \Psi\times \Phi' \times \Psi'} \S:(f,g,f',g')\mapsto \Phi (f)\times \Psi (g)\times \Phi' (f')\times \Psi' (g')$$}
along the functor
$$ \cat{Cospan}(\cat{Comp})\xrightarrow{\vp\omega} \A^{op}: \big((f,g), (f',g')\big) \mapsto (gf)\times _{R} (g'f'),$$
where $R$ is the codomain of $f$ and $f'$, corresponds to $(\op X_{\Phi}\circ \op X_{\Psi})\square (\op X_{\Phi'}\circ \op X_{\Psi'})$.
\end{enumerate}
\end{cor}

The domain and codomain of the natural transformation of Proposition \ref{prop:sigma}
can therefore be computed as left Kan extensions in a (noncommuting) diagram of the following form.
$$\xymatrix{\cat{Cospan}(\cat {Comp})\ar [d]_{\omega} \ar [rrr]^{\Phi\times \Psi\times \Phi'\times \Psi'}&&&\S\\
\cat {Comp}(\cat{Cospan})\ar [d]_{\vp} \ar [urrr]_{\qquad\Phi \times \Phi'\times \Psi\times \Psi'}\\
\A^{op}}$$
To prove the existence of the desired natural transformation
$$\sigma: (-\circ -)\square (-\circ -) \to \big((- \square -) \circ (- \square -)\big)(23) : \sseq ^{\times 4}\to \sseq,$$
where $(23)$ denotes the functor permuting the middle two copies of $\sseq$, it therefore suffices to construct a family of natural transformations $\mathfrak t$ from $\Phi\times \Psi\times \Phi' \times \Psi'$ to $(\Phi \times \Phi'\times \Psi\times \Psi')\omega$, which are themselves natural in $\Phi,\Phi', \Psi$, and $\Psi'$.    To establish the required associativity of $\sigma$, we must then prove that
\begin{equation} \label{eqn:t}
\xymatrix{\Phi\times \Psi\times \Phi'\times \Psi'\times \Phi''\times \Psi'' \ar [d]_{\id\times \mathfrak t}\ar [rr]^{\mathfrak t \times \id}&&(\Phi\times \Phi'\times \Psi\times \Psi')\omega\times \Phi''\times \Psi'' \ar [d]^{ \mathfrak t*(\omega \times \id)}\\
\Phi\times \Psi\times (\Phi'\times \Phi''\times \Psi'\times \Psi'')\omega\ar [rr]^{\mathfrak t*(\id\times \omega)}&&( \Phi\times \Phi'\times \Phi''\times \Psi \times \Psi'\times \Psi'')\omega^{(2)}}
\end{equation}
commutes, where $*$ denotes whiskering.

Before defining and studying $\mathfrak t$, we need one more important natural map.

\begin{notn}\label{notn:delta} Let $\Phi\in [\A^{{op}}, \S]_{mult}$.  Consider the following pullback diagram of set maps.
$$\xymatrix{T\times _{R}S'\ar[d]\ar [rr]^{g\times _{R}S'}&&S\times _{R}S'\ar [d]\ar [rr]&&S'\ar[d]_{f'}\\
T\ar [rr]^{g}&&S\ar [rr]^{f}&&R}
$$
There is a map from $\delta_{\Phi}: \Phi(g)\to \Phi(g\times _{R}S')$, which is essentially an iterated diagonal map and which is natural in $f$, $g$ and $f'$, as well as $\Phi$.  We can construct this map as follows.
Since $\Phi$ is multiplicative, 
$$\Phi (g)\cong \prod _{s} \Phi \Big( g^{-1}(s)\to \{s\}\Big)$$
and
\begin{align*}
\Phi (g\times _{R}S')&\cong \prod _{(s,s')\in S\times_{R}S' } \Phi \Big( (g\times _{R}S')^{-1}(s,s')\to \{(s,s')\}\Big)\\
&\cong \prod _{(s,s')\in S\times_{R}S' } \Phi \Big( g^{-1}(s)\times \{s'\}\to \{(s,s')\}\Big).
\end{align*}
For every $(s,s')\in S\times _{R}S'$, there is a morphism in $\cat A$
$$(\beta_{s,s'}, \alpha_{s,s'}): \Big( g^{-1}(s)\to \{s\}\Big) \to  \Big( g^{-1}(s)\times \{s'\}\to \{(s,s')\}\Big)$$
given by $\beta _{s,s'}(x)=(x,s')$ for all $x\in  g^{-1}(s)$.

Let 
$$\iota^{\Phi}_{s,s'}=\Phi (\beta_{s,s'}^{-1}, \alpha_{s,s'}^{-1}):  \Phi \Big( g^{-1}(s)\to \{s\}\Big) \to  \Phi \Big( g^{-1}(s)\times \{s'\}\to \{(s,s')\}\Big).$$
Using the decomposition into products, we can now define $\delta_{\Phi}: \Phi(g)\to \Phi(g\times _{R}S')$ by
$$\delta_{\Phi} \big((y_{s})_{s\in S}\big) = \big(\iota^{\Phi}_{s,s'}(y_{s})\big)_{(s,s')\in S\times _{R}S'}.$$
The naturality of $\delta_{\Phi}$ is clear.
\end{notn}

\begin{proof}[Proof of Proposition \ref{prop:sigma}]  Let $\Phi,\Phi', \Psi, \Psi'\in [\A^{{op}}, \S]_{mult}$.
There is a natural transformation 
$$\mathfrak t: \Phi\times \Psi\times \Phi' \times \Psi'\to (\Phi \times \Phi'\times \Psi\times \Psi')\omega$$
defined componentwise as follows.  Given $\big( (S\xrightarrow f R,T\xrightarrow g S), (S'\xrightarrow{f'} R, T'\xrightarrow {g'} S')\big)$, define
$$
\Phi(f)\times \Psi(g)\times \Phi'(f')\times \Psi'(g')\xrightarrow{\mathfrak t}\Phi(f)\times \Phi(f')\times \Psi(g\times_{R}S')\times \Psi'(S\times_{R}g')$$
by
$$\mathfrak t(x,y,x',y'')= \big(x,x',\delta_{\Phi'}(y), \delta_{\Psi'}(y')\big).$$ 
An easy computation, based on the coassociativity of the diagonal map of any simplicial set, shows that diagram (\ref{eqn:t}) commutes, as desired. 
\end{proof}

Now that we have chosen a specific natural transformation $\sigma$, we are ready to show that it is appropriately compatible with operad multiplication, as formulated in Theorem \ref{thm:tau}.

\begin{proof}[Proof of Theorem \ref{thm:tau}] For any pair of operads $\op P, \op Q$, let $\pi: \op P\square \op Q\to \op P\otimes \op Q$ denote  the natural composite
$$\op P\square \op Q\xrightarrow{\iota} \op P\circ \op Q\hookrightarrow \op P \coprod \op Q \xrightarrow\psi \op P\otimes \op Q,$$
where $\iota (p,q)=(p;q,...,q)$, the second morphism is the obvious inclusion into the coproduct of operads, and $\psi$ denotes the quotient map, so that $\pi(p,q)=p\otimes q$.

Define $\upsilon $ to be the composite
{\small$$(\op P\circ \op X\circ \op Q) \square (\op P'\circ \op X'\circ \op Q')\xrightarrow{(\sigma\circ \id)\sigma}(\op P\square \op P')\circ (\op X \square \op X') \circ (\op Q \square \op Q')\xrightarrow{\pi\circ \id \circ \pi} (\op P \otimes \op P') \circ (\op X \square \op X') \circ (\op Q\otimes \op Q').$$}
Proposition \ref{prop:sigma} implies that diagram (\ref{eqn:tau}) commutes if 
\begin{equation}\label{eqn:identity}\mu_{\op P\otimes \op Q}(\pi^{\circ 2})\sigma= \pi(\mu_{\op P}\square \mu_{\op Q}): \op P^{\circ 2}\square \op Q^{\circ 2}\to \op P\otimes \op Q
\end{equation}
for any operads $\op P$ and $\op Q$, so we now check this identity, in the coordinate-free framework set up above.
We begin by analyzing more closely the two sides of equation (\ref{eqn:identity}).

Let $\widetilde \sigma: (\op P\circ\op P)\square (\op Q \circ \op Q) \to \op P\circ (\op P\square \op Q) \circ \op Q$ denote the natural morphism of symmetric sequences specified by
$$\widetilde\sigma\big( (p; p_{1},...,p_{k}), (q;q_{1},..., q_{l})\big)= \big(p; (p_{1},q),..., (p_{k},q); q_{1},...,q_{l},..., q_{1},...,q_{l}\big),$$
where $p\in \op P(k)$, $q\in \op Q(l)$ and $p_{i}\in \op P(m_{i})$, $q_{j}\in \op Q(n_{j})$ for all $i$, $j$.  The associativity of the multiplication on $\op P\otimes \op Q$ and the fact that the inclusions $\op P\to \op P\otimes \op Q$ and $\op Q\to \op P\otimes \op Q$ are operad maps together imply easily that $\pi(\mu_{\op P}\square \mu_{\op Q})$ is equal to the composite
$$(\op P\circ \op P)\square (\op Q\circ \op Q) \xrightarrow {\widetilde \sigma} \op P\circ (\op P\square \op Q) \circ \op Q\xrightarrow {\id \circ \pi\circ \id} \op P\circ (\op P\otimes \op Q) \circ \op Q \xrightarrow{\lambda} \op P\otimes \op Q,$$
where the morphism $\lambda$ is constructed from the left action of $\op P$ and the right action of $\op Q$ on $\op P\otimes \op Q$.  It follows from the commutativity relation imposed on the Boardman-Vogt tensor product that 
$$\xymatrix{ \op P\circ (\op P\square \op Q) \circ \op Q\ar [dr]^{\quad \lambda(\id\circ \pi\circ\id) }\ar [dd]_{\id\circ sw\circ \id}^{\cong}\\
&\op P\otimes \op Q\\
 \op P\circ (\op Q\square \op P) \circ \op Q\ar [ur]_{\quad\lambda(\id\circ \pi\circ\id) }}$$
 commutes, where $sw:\op P\square\op Q \to \op Q\square \op P$ denotes the symmetry transformation of the matrix monoidal product, whence
 $$\pi(\mu_{\op P}\square \mu_{\op Q})=\lambda(\id\circ \pi\circ\id)(\id \circ {sw}\circ \id)\widetilde\sigma.$$
 
 Observe now that the diagram
 $$\xymatrix{ \op P\circ (\op Q\square \op P) \circ \op Q\ar [d]_{\id \circ \iota \circ\id }\ar[r]^(0.6){\lambda (\id\circ \pi\circ \id)}& \op P\otimes \op Q\\
 \op P\circ \op Q\circ \op P\circ \op Q\ar [r]^(0.45){\psi\circ \psi}&(\op P\otimes \op Q) \circ (\op P\otimes \op Q)\ar [u]_{\mu_{\op P\otimes \op Q}}}$$
 commutes, since the quotient map $\psi: \op P\coprod \op Q \to \op P\otimes \op Q$ is a morphism of operads, and the multiplication on $\op P\otimes \op Q$ is associative.  Consequently, 
 $$\pi(\mu_{\op P}\square \mu_{\op Q})=\mu_{\op P\otimes \op Q}(\psi\circ \psi)(\id\circ\iota\circ\id)(\id\circ {sw}\circ \id)\widetilde\sigma.$$
If we show moreover that 
\begin{equation}\label{eqn:final}(\iota \circ \iota)\sigma=(\id\circ \iota\circ \id)(\id\circ {sw}\circ \id)\widetilde\sigma,
\end{equation}
then we can conclude, since in that case
\begin{align*}
\mu_{\op P\otimes \op Q}(\pi\circ \pi)\sigma&=\mu_{\op P\otimes \op Q}(\psi\circ \psi)(\iota \circ \iota)\sigma\\
&=\mu_{\op P\otimes \op Q}(\psi\circ \psi)(\id\circ \iota\circ \id)(\id\circ {sw}\circ \id)\widetilde\sigma\\
&=\pi(\mu_{\op P}\square \mu_{\op Q}),
\end{align*}
as desired.  It remains thus to verify equation (\ref{eqn:final}).

We can translate equation (\ref{eqn:final}) into our coordinate-free framework as follows.   Let $\cat {Comp}(\cat{Comp})$ denote the subcategory of $\cat {Comp}\times \cat{Comp}$ the objects of which are 4-tuples $(f,g, h,k)$ of  composable maps.  Let $\cat A^{op}\vee \cat{Comp} \vee \cat A^{op}$ denote the subcategory of $(\cat A^{op})^{\times 4}$ the objects of which are sequences of set maps of the form
$$(S\xrightarrow f R,(T\xrightarrow g S,T'\xrightarrow {g'} S),U\xrightarrow k T\times_{S}T').$$
In the remainder of this proof we work with the following functors, in addition to those already defined:
\begin{itemize}
\item $\theta: \cat {Cospan}(\cat {Comp}) \to  \cat A^{op}\vee \cat{Comp} \vee \cat A^{op}$
specified by
$$\theta\Big( (S\xrightarrow f R, T\xrightarrow g S), (S'\xrightarrow {f'} R, T'\xrightarrow {g'} S')\Big)=\big( f, (S\times_{R}f', g), T\times _{R}g'\big),$$
\smallskip

\item $\chi:\cat A^{op}\vee \cat{Comp} \vee \cat A^{op}\to :\cat A^{op}\vee \cat{Comp} \vee \cat A^{op}$
specified by
$$\chi\Big(S\xrightarrow f R,(T\xrightarrow g S,T'\xrightarrow {g'}S),U\xrightarrow k T\times_{S}T'\Big)=\big(f,(g',g),h\big)$$
\smallskip

\item $\zeta:\cat A^{op}\vee \cat{Comp} \vee \cat A^{op} \to \cat {Comp}(\cat {Comp})$
specified by
$$\zeta\Big(S\xrightarrow f R,(T\xrightarrow g S,T'\xrightarrow {g'}S),U\xrightarrow h T\times_{S}T'\Big)=(f,g,T\times _{S}g', h),$$
\smallskip

\item $\kappa: \cat {Comp}(\cat{Cospan})\to \cat {Comp}(\cat{Comp})$ specified by
{\small $$\kappa\big( (S\xrightarrow f R, S'\xrightarrow{f'} R), (T\xrightarrow{g} S\times _{R}S', T'\xrightarrow{g'} S\times _{R}S')\big)=\big( (f, S\times_{R}f'), (g, T\times_{S\times_{R}S'}g')\big)$$}

\noindent and
\smallskip

\item $\operatorname{comp}^{(2)}:\cat {Comp}(\cat {Comp})\to  \cat A^{op}: (f,g,h,k)\mapsto khgf.$
\end{itemize}
\medskip

Recall the definition of the functor $\omega : \cat {Cospan}(\cat {Comp}) \to \cat {Comp}(\cat {Cospan})$ from the beginning of this section. The diagram of functors
\begin{equation}\label{eqn:functor-diagram}\xymatrix{\cat{Cospan}(\cat {Comp})\ar [d]_{\omega}\ar [rr]^{\theta}&&\cat A^{op}\vee \cat{Comp} \vee \cat A^{op}\ar [d]^{\chi}\\
\cat {Comp}(\cat {Cospan})\ar [d]_{\kappa}&&\cat A^{op}\vee \cat{Comp} \vee \cat A^{op}\ar [d]^{\zeta}\\
\cat {Comp}(\cat {Comp})\ar[dr]_{\operatorname{comp}^{(2)}}&& \cat {Comp}(\cat {Comp})\ar[dl]^{\operatorname{comp}^{(2)}}\\
&\cat A^{op}}
\end{equation}
commutes. To see this, fix any object $\Big( (S\xrightarrow f R, T\xrightarrow g S), (S'\xrightarrow {f'} R, T'\xrightarrow {g'} S')\Big)$ in $\cat {Cospan}(\cat {Comp})$, and consider the associated iterated pullback diagram.
$$\xymatrix{T\times _{R}T'\ar [d]^{T\times _{R}g'}\ar [rr]^{g\times_{R}T'}\ar[d]&&S\times_{R} T' \ar [rr]^(0.6){f\times_{R}T'}\ar [d]^{S\times_{R}g'}&&T'\ar [d]^{g'}\\
T\times _{R}S' \ar[d]^{T\times _{R}f'}\ar[rr]^{g\times_{R}S'}&& S\times _{R}S' \ar [d]^{S\times_{R}f'}\ar[rr]^(0.6){f\times_{R}S'}&&S'\ar [d]^{f'}\\
T\ar [rr]^g&&S\ar [rr]^f&&R}$$
The commutativity of this diagram implies that
\begin{align*}
\operatorname{comp}^{(2)}\kappa\omega\Big((f,g), (f',g')\Big)&=\operatorname{comp}^{(2)}\kappa\Big( (f,f'), (g\times _{R}S', S\times _{R}g')\Big)\\
&=\operatorname{comp}^{(2)}\Big(f, S\times _{R}f', g\times _{R}S', T\times _{R}g'\Big)\\
&=gf\times_{R}g'f'\\
&=\operatorname{comp}^{(2)}\Big(f, g, T\times_{R}f', T\times _{R}g'\Big)\\
&=\operatorname{comp}^{(2)}\zeta\Big( f, (g,S\times_{R}f'), T\times _{R}g'\Big)\\
&=\operatorname{comp}^{(2)}\zeta\chi \Big( f, (S\times_{R}f', g), T\times _{R}g'\Big)\\
&=\operatorname{comp}^{(2)}\zeta\chi\theta\Big((f,g), (f',g')\Big).
\end{align*}
We can also check easily that $\vp =\operatorname{comp}^{(2)}\kappa$ (cf. Corollary \ref{cor:comp-cospan}). 

The relevance of diagram (\ref{eqn:functor-diagram}) for verifying (\ref{eqn:final}) is a consequence of the following observations, which are based on Lemmas \ref{lem:monprods} and \ref{lem:comp-prod}. Under the identification of $\sseq$ and $[\cat A^{{op}}, \S]_{mult}$ of Lemma \ref{lem:sseq},
\begin{itemize}
\item   the left Kan extension of 
{\small $$\cat A^{op}\vee \cat{Comp} \vee \cat A^{op} \xrightarrow{\Phi\times \Psi\times \Psi' \times \Xi} \S:\big (f, (g,g'), h\big)\mapsto \Phi (f)\times \Psi (g)\times \Psi' (g')\times \Xi (h)$$}
along $\operatorname{comp}^{(2)}\zeta$ corresponds to $\op X_{\Phi}\circ  (\op X_{\Psi}\square \op X_{\Psi'})\circ \op X_{\Xi},$
\item  the left Kan extension of 
{\small $$\cat {Comp}(\cat{Comp}) \xrightarrow{\Phi\times \Psi\times \Xi \times \Omega} \S:\big (f, (g,g'), h\big)\mapsto \Phi (f)\times \Psi(g)\times \Xi (g')\times \Omega (h)$$}
along $\operatorname{comp}^{(2)}$ corresponds to $\op X_{\Phi}\circ  \op X_{\Psi}\circ \op X_{\Psi'}\circ \op X_{\Xi}.$
\end{itemize}

For any $\Phi, \Psi, \Phi', \Psi'\in [\cat A^{{op}}, \S]_{mult}$, consider the following augmented version of commutative diagram (\ref{eqn:functor-diagram}).
\begin{equation}\label{eqn:natltransfdiag}\xymatrix{\cat{Cospan}(\cat {Comp})\ar [dd]_{\omega}\ar [ddrr]^{\Phi\times \Psi\times \Phi'\times \Psi'}\ar [rrrr]^{\theta}&&&&\cat A^{op}\vee \cat{Comp} \vee \cat A^{op}\ar [ddll]_{\Phi\times \Psi\times \Phi'\times \Psi'}\ar [dd]^{\chi}\\
\\
\cat {Comp}(\cat {Cospan})\ar [dd]_{\kappa}\ar [rr]^(0.6){\Phi\times \Phi'\times \Psi\times \Psi'}&&\S &&\cat A^{op}\vee \cat{Comp} \vee \cat A^{op}\ar[ll]_(0.6){\Phi\times \Phi'\times \Psi\times \Psi'}\ar [dd]^{\zeta}\\
\\
\cat {Comp}(\cat {Comp})\ar[drr]_{\operatorname{comp}^{(2)}}\ar [uurr]_{\Phi\times \Phi'\times \Psi\times \Psi'}&&&& \cat {Comp}(\cat {Comp})\ar[dll]^{\operatorname{comp}^{(2)}}\ar [uull]^{\Phi\times \Phi'\times \Psi\times \Psi'}\\
&&\cat A^{op}}
\end{equation}

Each triangle in the diagram above can be filled in with a natural tranformation, to create a commuting diagram of natural transformations.  First, it is obvious that 
$$\Phi\times \Psi\times \Phi'\times \Psi' =(\Phi\times \Phi'\times \Psi\times \Psi')\chi,$$ 
which corresponds to the symmetry of the matrix monoidal product. We know moreover that there is a natural transformation $$\mathfrak t: \Phi\times \Psi\times \Phi'\times \Psi' \to (\Phi\times \Phi'\times \Psi\times \Psi')\omega,$$ 
giving rise to the natural transformation $\sigma$.  By very similar arguments, using the natural maps $\delta_{(-)}$ (cf. Notation \ref{notn:delta}), one can show easily that there are natural transformations
$$\widetilde{\mathfrak t}: \Phi\times \Psi\times \Phi'\times \Psi' \to (\Phi\times \Psi\times \Phi'\times \Psi')\theta,$$
corresponding to $\widetilde \sigma$,
$$\mathfrak i: \Phi\times \Phi'\times \Psi\times \Psi'\to (\Phi\times \Phi'\times \Psi\times \Psi')\kappa,$$
corresponding to $\iota\circ \iota$, and
$$\widetilde{\mathfrak \i}:\Phi\times \Phi'\times \Psi\times \Psi'\to (\Phi\times \Phi'\times \Psi\times \Psi')\zeta,$$
corresponding to $\id\circ \iota \circ \iota$, such that
$$(\mathfrak i*\omega) \mathfrak t = (\widetilde{\mathfrak \i}*\chi\theta)\widetilde{\mathfrak t},$$
from which we can conclude that (\ref{eqn:final}) does indeed hold.
\end{proof}

 \bibliographystyle{amsplain}
\bibliography{divpower}
\end{document}